\documentclass[sn-mathphys]{sn-jnl}
\usepackage{amssymb}
\usepackage{amsmath}
\usepackage{natbib}
\usepackage{hyperref}
\jyear{2021}
\theoremstyle{thmstyleone}

\newtheorem{definition}{Definition}
\newtheorem{example}{Example}
\newtheorem{lemma}{Lemma}
\newtheorem{remark}{Remark}
\newtheorem{theorem}{Theorem}

\def\R{\mathbb{R}}

\begin{document}
\title[Article title]{Extended convergence analysis of the Scholtes-type regularization for cardinality-constrained optimization problems}
\author[1]{\fnm{Sebastian} \sur{Lämmel}}\email{sebastian.laemmel@mathematik.tu-chemnitz.de}
\author[1]{\fnm{Vladimir} \sur{Shikhman}}\email{vladimir.shikhman@mathematik.tu-chemnitz.de}
\affil[1]{\orgdiv{Faculty of Mathematics}, \orgname{TU Chemnitz}, \orgaddress{\street{Reichenhainer Straße 41}, \city{Chemnitz}, \postcode{09126}, \country{Germany}}}

\abstract{ We extend the convergence analysis of the Scholtes-type regularization method for cardinality-constrained optimization problems. Its behavior is clarified in the vicinity of saddle points, and not just of minimizers as it has been done in the literature before. This becomes possible by using as an intermediate step the recently introduced regularized continuous reformulation of a cardinality-constrained optimization problem. We show that the Scholtes-type regularization method is well-defined locally
around a nondegenerate T-stationary point of this regularized continuous reformulation. Moreover,
the nondegenerate Karush-Kuhn-Tucker points of the corresponding Scholtes-type regularization converge to a T-stationary point having the same index, i.e. its topological type persists. Overall, we conclude that the global structure of the regularized continuous reformulation and its Scholtes-type regularization essentially coincide.}

\maketitle

\setlength{\abovedisplayskip}{5pt}
\setlength{\abovedisplayshortskip}{5pt}
\setlength{\belowdisplayskip}{5pt}
\setlength{\belowdisplayshortskip}{5pt}

\section{Introduction}\label{sec1}

In nonconvex optimization Scholtes-type regularization methods became popular since the seminal paper \cite{scholtes:2001}. Typically, hard nonsmooth constraints are relaxed by means of a parameter. 
Then, Karush-Kuhn-Tucker points of the induced nonlinear programs need to be computed.
They are finally shown to converge towards some suitably defined stationary points of the original optimization problem as the regularization parameter tends to zero. Scholtes-type regularization methods for mathematical programs with complementarity (MPCC), vanishing (MPVC), switching (MPSC), and orthogonality type constrains (MPOC) were examined along these lines in the literature so far, see e.g. \cite{scholtes:2001}, \cite{izmailov:2009}, \cite{kanzow:2021}, and \cite{laemmel:mpoc} for further details, respectively. 

In this paper, we study the Scholtes-type regularization method for the class of cardinality-constrained optimization problems:
\[
\mbox{CCOP}: \quad
\min_{x} \,\, f(x)\quad \mbox{s.\,t.} \quad h(x)=0, \quad g(x)\ge 0, \quad \left\|x\right\|_0\le s
\]
with the feasible set given by equality, inequality, and cardinality constraints, where the so-called zero ”norm” $\left\|x\right\|_0 = \left\vert \left\{i \in \{1,\ldots, n\}\; \vert\; x_i\ne 0\right\}\right\vert$ is counting non-zero entries of $x$.
Here, we assume that the objective function $f$, as well as the equality and inequality constraints 
$h=\left(h_p, p \in P\right)$,
$g=\left(g_q, q \in Q\right)$ are twice continuously differentiable, and $s \in \{0,1,\ldots, n-1\}$ is an integer. 
In order to arrive at the Scholtes-type regularization, the so-called continuous reformulation of CCOP 
from \cite{burdakov:2016} is helpful:
\begin{equation}
     \label{eq:relax}
     \begin{array}{rl}
    \displaystyle \min_{x,y} \,\, f(x) \quad \mbox{s.\,t.} \quad 
   &h(x)=0, \quad g(x)\ge 0, \quad \displaystyle \sum_{i=1}^{n} y_i \geq  n - s,\\
   & x_i y_i =0,  \quad 0 \leq y_i \leq 1, \quad 
         i=1, \ldots, n.
   \end{array}
\end{equation}
As pointed out there, $\bar x$ solves CCOP if and only if there exists a vector $\bar y$ such that $\left(\bar x, \bar y\right)$ solves (\ref{eq:relax}). 
In order to tackle (\ref{eq:relax}) numerically, \cite{bucher:2018} suggests to regularize the orthogonality type constraints by using the Scholte's idea, cf. \cite{scholtes:2001}: 
\begin{equation}
     \label{eq:relax-scholtes}
          \begin{array}{rl}
   \displaystyle  \min_{x,y} \,\, f(x)\quad \mbox{s.\,t.} 
   &h(x)=0, \quad g(x)\ge 0, \quad \displaystyle \sum_{i=1}^{n} y_i \geq  n - s, \\
   & -t\le x_i y_i \le t,\quad  0 \leq y_i \leq 1, \quad
         i=1, \ldots, n,
   \end{array}
\end{equation}
where $t>0$. Further in \cite{branda:2018}, the authors prove that -- under some suitable constraint qualification and second-order sufficient condition -- the Scholtes-type regularization method is well-defined locally around a minimizer of (\ref{eq:relax}). Moreover, the Karush-Kuhn-Tucker points of (\ref{eq:relax-scholtes}) converge to an S-stationary point of (\ref{eq:relax}) whenever $t \rightarrow 0$. 

Our goal is to extend the convergence analysis of the Scholtes-type regularization method beyond the case of minimizers of (\ref{eq:relax}), but also for all kinds of its saddle points. By doing so, we intend to relate the indices of nondegenerate Karush-Kuhn-Tucker points of the Scholtes-type regularization with those of T-stationary points of the regularized continuous reformulation. Here, nondegeneracy refers to some tailored versions of linear independence constraint qualification, strict complementarity and second-order regularity. Assuming nondegeneracy, Karush-Kuhn-Tucker points and T-stationary points can be classified according to their quadratic and T-index, respectively. The index encodes the local structure of the optimization problem under consideration in algebraic terms and its global structure in the sense of Morse theory, see \cite{Jongen:2000} and \cite{laemmel:reform}. We note that for our purpose we need to preliminarily regularize the continuous reformulation (\ref{eq:relax}). The reason is that all T-stationary points of (\ref{eq:relax}) -- considered as an MPOC instance -- turn out to be degenerate, cf. \cite{laemmel:mpoc}. To overcome this obstacle, it has been suggested in \cite{laemmel:reform} not only to linearly perturb the objective function in (\ref{eq:relax}) with respect to $y$-variables, but also to additionally relax the upper bounds on them. 
As for our main results, the Scholtes-type regularization method proves to be well-defined locally around a nondegenerate T-stationary point of the regularized continuous reformulation. Moreover, the nondegenerate Karush-Kuhn-Tucker points of its Scholtes-type regularization converge to a T-stationary point having the same index.

The article is organized as follows. In Section 2 we discuss some preliminary results on the regularized continuous reformulation of CCOP. 
Section 3 is devoted to the extended convergence analysis of its Scholtes-type regularization.

Our notation is standard. The cardinality of a finite set $A$ is denoted by $\vert A \vert$. The $n$-dimensional Euclidean space is denoted by $\mathbb{R}^n$ with the coordinate vectors $e_i,i= 1,\ldots, n$. The vector consisting of ones is denoted by $e$. Given a twice continuously differentiable function $f:\mathbb{R}^n\rightarrow \mathbb{R}$, $\nabla f$ denotes its gradient, and $D^2f$ stands for its Hessian.

\section{Regularized continuous reformulation}
We associate with CCOP the regularized continuous reformulation from \cite{laemmel:reform}:
\[
     \begin{array}{rl}
    \displaystyle \mathcal{R}(c,\varepsilon): \quad \min_{x,y} \,\, f(x) +c^Ty\quad \mbox{s.\,t.} 
   &h(x)=0, \quad g(x)\ge 0, \quad \displaystyle \sum_{i=1}^{n} y_i \geq  n - s, \quad \\
   &x_i y_i =0, \quad 0 \leq y_i \leq 1+\varepsilon,  \quad i=1, \ldots, n,
   \end{array}
\]
where the components of $c \in \R^n$ $c$ are positive and pairwise different, and $0 < \varepsilon \leq \frac{1}{n-s}$.
Given a feasible point  $(\bar x,\bar y)$ of $\mathcal{R}$, we define the index sets which correspond to the orthogonality type constraints $x_i y_i=0$, $y_i \geq 0$:
    \[
     \begin{array}{c}
        a_{01}\left(\bar x,\bar y\right)=\left\{i\,\left\vert\, \bar x_i=0,\bar y_i>0\right.\right\}, \quad 
    a_{10}\left(\bar x,\bar y\right)=\left\{i\,\left\vert\, \bar x_i\ne 0,\bar y_i=0\right.\right\}, \\
    a_{00}\left(\bar x,\bar y\right)=\left\{i\,\left\vert\, \bar x_i=0,\bar y_i=0\right.\right\}.       
     \end{array}
    \]
The index sets of the active inequality constraints will be denoted by
\[
 Q_0(\bar x)=\left\{q \in Q \, \left\vert\, g_q(\bar x)=0\right.\right\}, \quad\mathcal{E}(\bar y)=\left\{i \,\left\vert\,\bar y_i=1+\varepsilon\right.\right\}.
\]
The regularized continuous reformulation $\mathcal{R}$ is a special case of MPOC. The latter class was examined in \cite{laemmel:mpoc}, where the MPOC-tailored linear independence constraint qualification and the notion of (nondegenerate) T-stationary points with the corresponding T-index were introduced. 
In \cite{laemmel:reform}, these concepts were applied to the regularization $\mathcal{R}$. 

\begin{definition}[MPOC-LICQ, \cite{laemmel:reform}]
We say that a feasible point $(\bar x,\bar y)$ of $\mathcal{R}$ satisfies the MPOC-tailored linear independence constraint qualification (MPOC-LICQ) if the following vectors are linearly independent:
\[
\begin{array}{c}
\begin{pmatrix}
\nabla h_p(\bar x)\\0
\end{pmatrix}, p \in P, \quad 
\begin{pmatrix}
\nabla g_q(\bar x)\\0
\end{pmatrix}, q \in Q_0(\bar x), \quad  
\begin{pmatrix}
0\\
e_i
\end{pmatrix},i\in \mathcal{E}(\bar y), \quad
\begin{pmatrix}
0\\
e
\end{pmatrix}
\mbox{ if } \sum\limits_{i=1}^{n} \bar y_i = n - s,
\\
\begin{pmatrix}
e_i\\
0
\end{pmatrix}, i \in a_{01}\left(\bar x,\bar y\right)\cup a_{00}\left(\bar x,\bar y\right), \quad 
\begin{pmatrix}
0\\
e_i
\end{pmatrix}, i \in a_{10}\left(\bar x, \bar y\right)\cup a_{00}\left(\bar x, \bar y\right).
\end{array}
\]
\end{definition}

\begin{definition}[T-stationary point, \cite{laemmel:reform}]
\label{def:t-stat}
A feasible point $(\bar x,\bar y)$ of $\mathcal{R}$ is called T-stationary if there exists multipliers 
\[
\begin{array}{l}
\bar \lambda_p, p \in P,
\bar \mu_{1,q}, q\in Q_0(\bar x),\bar \mu_{2,i},i \in \mathcal{E}(\bar y),\bar \mu_3,\\ \bar \sigma_{1,i}, i \in a_{01}\left(\bar x,\bar y\right),\bar \sigma_{2,i}, i \in a_{10}\left(\bar x,\bar y\right),\bar \varrho_{1,i},\bar \varrho_{2,i}, i \in a_{00}\left(\bar x, \bar y\right),
\end{array}
\]
such that the following conditions hold:
\begin{equation}
   \label{eq:tstat-1} 
   \begin{array}{rcl}
  \begin{pmatrix}
   \nabla f(\bar x)\\
   c
  \end{pmatrix}&=& \displaystyle\sum\limits_{p \in P}\bar \lambda_p 
  \begin{pmatrix}
  \nabla h_p(\bar x)\\
  0
  \end{pmatrix}+
    \sum\limits_{q \in Q_0(\bar x)}\bar \mu_{1,q} 
    \begin{pmatrix}
    \nabla g_q(\bar x)\\
    0
    \end{pmatrix}-
    \sum\limits_{i\in\mathcal{E}(\bar y)} \bar \mu_{2,i} \begin{pmatrix}
    0\\
    e_i
\end{pmatrix}\\ 
    && +
\bar \mu_3 \begin{pmatrix}
0\\
e
\end{pmatrix}\displaystyle +\sum\limits_{i \in a_{01}\left(\bar x,\bar y\right)} \bar \sigma_{1,i}
    \begin{pmatrix}
    e_{ i}\\
    0
    \end{pmatrix}
    +\sum\limits_{i \in a_{10}\left(\bar x,\bar y \right)} \bar \sigma_{2,i}    \begin{pmatrix}
    0\\
    e_{ i}
    \end{pmatrix}\\ 
    &&\displaystyle+\sum\limits_{i \in a_{00}\left(\bar x, \bar y\right)} \left(\bar \varrho_{1,i}
    \begin{pmatrix}
   e_{ i}\\
    0
    \end{pmatrix}
   +\bar \varrho_{2,i}
    \begin{pmatrix}
    0\\
    e_{ i}
    \end{pmatrix}\right), \end{array}
\end{equation}
\begin{equation}
   \label{eq:tstat-2} \bar \mu_{1,q} \ge 0,q\in Q_0\left(\bar x\right),
   \bar \mu_{2,i} \ge 0, i\in \mathcal{E}(\bar y),\bar \mu_3\ge 0, \bar\mu_3\left(
   \sum\limits_{i=1}^{n} \bar y_i -(n-s)\right)=0,
\end{equation}
\begin{equation}
   \label{eq:tstat-3} \bar \varrho_{1,i}=0 \mbox{ or }\bar \varrho_{2,i}\le 0, i \in a_{00}\left(\bar x,\bar y\right).
\end{equation}
\end{definition}
We define the appropriate Lagrange function:
\[
        \begin{array}{rcl}
        L^\mathcal{R}(x,y)&=& \displaystyle
   f(x)+
   c^T y -
        \sum\limits_{p \in P}\bar \lambda_p
  h_p( x)-
    \sum\limits_{q \in Q_0(\bar x)}\bar \mu_{1,q} 
    g_q(x) \\  & &\displaystyle+\sum\limits_{i\in\mathcal{E}(\bar y)} \bar \mu_{2,i} 
   \left( y_i- (1+\varepsilon)\right)
-
\bar \mu_3 
 \left(\sum\limits_{i=1}^{n} y_i - (n-s)\right)
   \\ 
    &&\displaystyle -\sum\limits_{i \in a_{01}\left(\bar x, \bar y\right)} \bar \sigma_{1,i}
     x_{ i}
    -\sum\limits_{i \in a_{10}\left(\bar x, \bar y\right)} \bar \sigma_{2,i}   
     y_{ i} -\sum\limits_{i \in a_{00}\left(\bar x, \bar y \right)} \left(\bar \varrho_{1,i}
    x_{ i}
   +\bar \varrho_{2,i}
     y_{ i}
    \right).
    \end{array}
\]
Moreover, we set for the corresponding tangent space:

\[
    \mathcal{T}^{\mathcal{R}}_{(\bar x,\bar y)} =\left\{
\xi \in \R^{2n}\,\left\vert\, \begin{array}{l} \begin{pmatrix}
Dh_p(\bar x), 0\end{pmatrix} \xi=0, p \in P, 
\begin{pmatrix}
Dg_q(\bar x), 0\end{pmatrix}\xi=0,q \in Q_0(\bar x),\\
\begin{pmatrix}
0,e_i
\end{pmatrix}\xi=0, i\in \mathcal{E}(\bar y), 
\begin{pmatrix}
0,e
\end{pmatrix}\xi=0 \mbox{ if } \displaystyle \sum_{i=1}^{n} \bar y_i = n - s,
\\
\begin{pmatrix}
e_i,0
\end{pmatrix}\xi=0, i \in a_{00}(\bar x,\bar y) \cup a_{01}(\bar x,\bar y),\\
\begin{pmatrix}
0,e_i
\end{pmatrix}\xi=0, i \in a_{00}(\bar x,\bar y) \cup a_{10}(\bar x,\bar y)
\end{array}
\right.\right\}.
\]

\begin{definition}[Nondegenerate T-stationary point, \cite{laemmel:reform}]
A T-stationary point $(\bar x,\bar y)$ of $\mathcal{R}$ with multipliers $(\bar \lambda, \bar \mu, \bar \sigma, \bar \varrho)$ is called nondegenerate if

NDT1: MPOC-LICQ holds at $(\bar x,\bar y)$,

NDT2: $\bar \mu_{1,q}>0$, $q \in Q_0\left(\bar x\right)$, $\bar \mu_{2,i}>0$, $i \in \mathcal{E}\left(\bar y\right)$, and $\bar \mu_3>0$ if  $\sum\limits_{i=1}^{n} \bar  y_i =n-s$,

NDT3: $\bar\varrho_{1,i}\ne 0$ and $\bar\varrho_{2,i}< 0$, $i\in a_{00}\left(\bar x,\bar y\right)$,
    
NDT4: the matrix $D^2 L^{\mathcal{R}}(\bar x,\bar y)\restriction_{\mathcal{T}^\mathcal{R}_{(\bar x,\bar y)}}$ is nonsingular.

\noindent
For a nondegenerate T-stationary point we eventually use an additional condition:

NDT6: $\bar \sigma_{1,i}\ne 0$, $i\in a_{01}(\bar x, \bar y)$.
\end{definition}

\begin{definition}[T-index, \cite{laemmel:reform}]
Let $(\bar x,\bar y)$ be a nondegenerate T-stationary point of $\mathcal{R}$ with unique multipliers $\left(\bar \lambda,\bar \mu, \bar \sigma,\bar \varrho\right)$. The number of negative eigenvalues of the matrix $D^2 L^{\mathcal{R}}(\bar x,\bar y)\restriction_{\mathcal{T}^\mathcal{R}_{(\bar x,\bar y)}}$ is called its quadratic index ($QI$). The cardinality of $a_{00}\left(\bar x,\bar y\right)$ is called the biactive index ($BI$) of $(\bar x,\bar y)$. We define the T-index ($TI$) as the sum of both, i.\,e. $TI=QI+BI$.
\end{definition}

The nondegeneracy conditions NDT1-NDT4 are tailored for $\mathcal{R}$. Note that NDT2 corresponds to the strict complementarity and NDT4 to the second-order regularity as they are typically defined in the context of nonlinear programming. NDT1 substitutes the usual linear independence constraint qualification. NDT3 is new and says that the multipliers corresponding to biactive orthogonality type constraints must not vanish. With a nondegenerate T-stationary point $(\bar x,\bar y)$ a T-index can be associated. The T-index captures the structure of $\mathcal{R}$ locally around $(\bar x,\bar y)$ and defines the type of a T-stationary point, see \cite{laemmel:reform} for details. In particular, nondegenerate minimizers of $\mathcal{R}$ are characterized by a vanishing T-index. If the T-index does not vanish, we get all kinds of saddle points.
We note that nondegenerate M-stationary points of CCOP, see e.g. \cite{laemmel:ccop}, naturally correspond to nondegenerate T-stationary points of $\mathcal{R}$ and vice versa. As shown in \cite{laemmel:reform}, also their M- and T-indices coincide. Thus, the regularized continuous reformulation $\mathcal{R}$ can be likewise studied instead of (\ref{eq:relax}).

Next Lemma \ref{lem:a01} provides insights into the structure of auxiliary $y$-variables corresponding to a T-stationary point of $\mathcal{R}$. Its proof will be useful in what follows as well.

\begin{lemma}[Auxiliary $y$-variables in $\mathcal{R}$, \cite{laemmel:reform}]
\label{lem:a01} 
Let $(\bar x,\bar y)$ be a T-stationary point of $\mathcal{R}$, then it holds:

a) the summation inequality constraint is active, i.\,e. $ \sum\limits_{i=1}^{n} \bar y_i =n - s$,

b) the index set $a_{01}(\bar x,\bar y)$ consists of exactly $n-s$ elements,
   
c) $n-s-1$ components of $\bar y$ are equal to $1+\varepsilon$, one component is equal to $1-(n-s-1)\varepsilon$, and $s$ remaining components vanish.
\end{lemma}
\begin{proof}
 a)
Let $(\bar x,\bar y)$ be a T-stationary point of $\mathcal{R}$
and $\displaystyle \sum_{i=1}^{n} \bar y_i >n - s$. Then, there exist multipliers
$(\bar \lambda, \bar \mu, \bar \sigma, \bar \varrho)$
such that (\ref{eq:tstat-1})--(\ref{eq:tstat-3}) are fulfilled.
Since $\bar \mu_3=0$, we have that the $(n+i)$-th row of (\ref{eq:tstat-1}) reads as
\[
c_i=\left\{
\begin{array}{ll}
     -\bar \mu_{2,i},&\mbox{for }i \in \mathcal{E}(\bar y),  \\
     \bar \sigma_{2,i}, &\mbox{for }i \in a_{10}\left(\bar x,\bar y\right),\\
     \bar \varrho_{2,i}, &\mbox{for }i\in a_{00}\left(\bar x,\bar y\right),\\
     0,&\mbox{else}.
\end{array}\right.
\]
Due to $c>0$, the sets $a_{01}\left(\bar x,\bar y\right)$ and $\mathcal{E}(\bar y)$ have to be equal and, moreover,
due to $\bar \mu_{2,i}\ge 0$, they have to be empty. But then, clearly, $ \sum_{i=1}^{n} \bar  y_i =0<n - s$, a contradiction.

b)
Since $(\bar x,\bar y)$ is a T-stationary point, there exist
$(\bar \lambda, \bar \mu, \bar \sigma, \bar \varrho)$
such that (\ref{eq:tstat-1})--(\ref{eq:tstat-3}). By the proof of statement a) we can conclude that
$\bar \mu_3 >0$. Hence,
the $(n+i)$-th row reads as
\begin{equation}
    \label{eq:k-throw}
c_i=\left\{
\begin{array}{ll}
     -\bar \mu_{2,i}+\bar \mu_3,&\mbox{for }i \in \mathcal{E}(\bar y),  \\
     \bar \sigma_{2,i}+\bar \mu_3, &\mbox{for }i \in a_{10}\left(\bar x,\bar y\right),\\
     \bar \varrho_{2,i}+\bar \mu_3, &\mbox{for }i\in a_{00}\left(\bar x,\bar y\right),\\
     \bar \mu_3,&\mbox{else}.
\end{array}\right.
\end{equation}
Let us assume that the index set $a_{01}(\bar x,\bar y)$ consists of fewer than $n-s$ elements. Then, we have by using the assumption on $\varepsilon$:
\[ \sum_{i=1}^{n} \bar y_i \le (n-s-1)\cdot (1+\varepsilon)\le n-s-1+\frac{n-s-1}{n-s} <n - s,
\]
a contradiction to feasibility. 
Let us assume that the index set $a_{01}(\bar x,\bar y)$ consists of more than $n-s$ elements instead.
Since the components of $c$ are assumed to be pairwise different, we see from (\ref{eq:k-throw}) that there exists at most one element in $a_{01}\left(\bar x,\bar y\right)\backslash \mathcal{E}(\bar y)$ and, consequently, there are at least $n-s$ elements in $\mathcal{E}(\bar y)$. Therefore, we have:
\[
\sum_{i=1}^{n} \bar y_i \geq (n-s)\cdot (1+\varepsilon) > n-s,
\]
which contradicts a). 

c) Due to b), $a_{01}(\bar x,\bar y)$ consists of exactly $n-s$ elements. We conclude as in b) that there is at most one element in $a_{01}\left(\bar x,\bar y\right)\backslash \mathcal{E}(\bar y)$. In view of statement a), $\mathcal{E}(\bar y)$ cannot consist of $n-s$ elements and, thus, must consist of $n-s-1$ elements. Hence, the statement follows immediately.
\end{proof}

\section{Scholtes-type regularization}

Let us now regularize the orthogonality type constraints in $\mathcal{R}$ by using the Scholte's idea, cf. \cite{scholtes:2001}: 
\[
     \begin{array}{rl}
    \displaystyle \mathcal{S}(t): \quad \min_{x,y} \,\, f(x) +c^Ty\quad \mbox{s.\,t.} 
   &h(x)=0, \quad g(x)\ge 0, \quad\displaystyle \sum_{i=1}^{n} y_i \geq  n - s, \\
   & -t\le x_i y_i \le t, \quad 0 \leq y_i \leq 1+\varepsilon,  \quad 
         i=1, \ldots, n,
   \end{array}
\]
where $t>0$. 
Note that $\mathcal{S}$ from above falls into the scope of nonlinear programming, i.e. just smooth equality and inequality constraints are present. The notation for the sets 
$Q_0(x)$ and $\mathcal{E}(y)$, which were used for $\mathcal{R}$, will be used here again.
Furthermore, we define for a feasible point $\left( x, y\right)$ of $\mathcal{S}$ the index set of vanishing $y$-components as well as the index sets of active relaxed orthogonality type constraints:
\[
\mathcal{N}( y)=\left\{i\, \left\vert\,  y_i=0\right.\right\},
\quad
\mathcal{H}^{\ge}\left(x, y\right)=\left\{i\,\left\vert \,x_i  y_i=-t\right.\right\}, \mathcal{H}^{\le}\left( x, y\right)=\left\{i\,\left\vert \, x_i  y_i=t\right.\right\}.
\]
We also eventually use the following index sets:
\[
\mathcal{H}\left(x, y\right)=\mathcal{H}^\geq\left(x, y\right) \cup \mathcal{H}^\leq\left( x, y\right), \quad 
\mathcal{O}\left( x, y\right)=\left(\mathcal{E}\left( y\right)\cup \mathcal{N}\left(y\right)\cup \mathcal{H}\left( x, y\right)\right)^c.
\]

 For the sake of completeness we state the linear independence constraint qualification for the nonlinear programming problem $\mathcal{S}$.

\begin{definition}[LICQ]
We say that a feasible point $(x, y)$ of $\mathcal{S}$ satisfies the linear independence constraint qualification (LICQ) if the following vectors are linearly independent:
\[
\begin{array}{c}
      \begin{pmatrix}
\nabla h_p(x)\\0
\end{pmatrix}, p \in P, \quad 
\begin{pmatrix}
\nabla g_q(x)\\0
\end{pmatrix}, q \in Q_0(x), \quad  
\begin{pmatrix}
0\\
e_i
\end{pmatrix},i\in \mathcal{E}(y),
\\ \begin{pmatrix}
0\\
e
\end{pmatrix}
\mbox{ if } \sum\limits_{i=1}^{n} y_i = n - s, \quad
\begin{pmatrix}
y_i e_i\\
x_i e_i
\end{pmatrix}, i \in \mathcal{H}\left(x,y\right), \quad 
\begin{pmatrix}
0\\
e_i
\end{pmatrix}, i \in \mathcal{N}(y).
\end{array}
\]
\end{definition}

Let us relate MPOC-LICQ for $\mathcal{R}$ with LICQ for $\mathcal{S}$. 

\begin{theorem}[MPOC-LICQ vs. LICQ]
\label{thm:mpoc-licq}
Let a feasible point $(\bar x, \bar y)$ of $\mathcal{R}$ fulfill MPOC-LICQ. Then, LICQ holds at all feasible points $(x,y)$ of $\mathcal{S}$ for all sufficiently small $t$, whenever they are sufficiently close to $(\bar x,\bar y)$. 
\end{theorem}
\setlength{\abovedisplayskip}{0pt}
\begin{proof} 
Let us contrarily assume that there exists a sequence of feasible points $\left(x^t,y^t\right)$ of $\mathcal{S}$ violating LICQ, which converges to $(\bar x, \bar y)$ for $t \rightarrow 0$. Additionally, suppose that along some subsequence, which we index by $t$ again, it holds $\sum\limits_{i=1}^{n} y^t_i = n - s$. Then, we have 
$\sum\limits_{i=1}^{n} \bar y_i = n - s$. Due to MPOC-LICQ at $\left(\bar x,\bar y\right)$ as well as continuity of $\nabla h$ and $\nabla g$, we have that for $t$ sufficiently small all multipliers $\bar \lambda^t,\bar \mu^t,\bar \sigma^t,\bar \varrho^t$ in the following equation vanish:

\begin{equation}
    \label{eq:MPOCLICQxt}
   \begin{array}{rcl}
  \begin{pmatrix}
   0\\
  0
  \end{pmatrix}&=& \displaystyle\sum\limits_{p \in P}\bar \lambda^t_p 
  \begin{pmatrix}
  \nabla h_p\left(x^t\right)\\
  0
  \end{pmatrix}+
    \sum\limits_{q \in Q_0(\bar x)}\bar \mu^t_{1,q} 
    \begin{pmatrix}
    \nabla g_q\left(x^t\right)\\
    0
    \end{pmatrix}-
    \sum\limits_{i\in\mathcal{E}(\bar y)} \bar \mu^t_{2,i} \begin{pmatrix}
    0\\
    e_i
\end{pmatrix}\\
    && \displaystyle +
\bar \mu^t_3 \begin{pmatrix}
0\\
e
\end{pmatrix}+\sum\limits_{i \in a_{01}\left(\bar x,\bar y\right)} \bar \sigma^t_{1,i}
    \begin{pmatrix}
    y_{i}^te_{ i}\\
    x_{i}^te_{ i}
    \end{pmatrix}
    +\sum\limits_{i \in a_{10}\left(\bar x,\bar y \right)} \bar \sigma^t_{2,i}    \begin{pmatrix}
    y_{i}^te_{ i}\\
    x_{i}^te_{ i}
    \end{pmatrix}\\
    &&\displaystyle+\sum\limits_{i \in a_{00}\left(\bar x, \bar y\right)} \left(\bar \varrho^t_{1,i}
    \begin{pmatrix}
   e_{ i}\\
    0
    \end{pmatrix}
   +\bar \varrho^t_{2,i}
    \begin{pmatrix}
    0\\
    e_{ i}
    \end{pmatrix}\right)
    . \end{array}
\end{equation}
Moreover, due to the violation of LICQ at $\left(x^t,y^t\right)$, there exist multipliers $\lambda^t,\mu^t,\eta^t,\nu^t$, not all vanishing, with
\[
   \begin{array}{rcl}
  \begin{pmatrix}
  0\\
   0
  \end{pmatrix}&=& \displaystyle \sum\limits_{p \in P} \lambda^t_p 
  \begin{pmatrix}
  \nabla h_p( x^t)\\
  0
  \end{pmatrix}+
    \sum\limits_{q \in Q_0( x^t)} \mu^t_{1,q} 
    \begin{pmatrix}
    \nabla g_q( x^t)\\
    0
    \end{pmatrix}+
    \sum\limits_{i\in\mathcal{E}( y^t)}  \mu^t_{2,i} \begin{pmatrix}
    0\\
     e_i
\end{pmatrix}\\ 
&& \displaystyle+
 \mu^t_3 \begin{pmatrix}
0\\
e
\end{pmatrix} 
+ \sum\limits_{i\in\mathcal{H}( x^t,  y^t)}  \eta^t_{i} \begin{pmatrix}
    y^t_ie_i\\
   x^t_ie_i
\end{pmatrix}+
    \sum\limits_{i\in\mathcal{N}( y^t)}  \nu^t_{i} \begin{pmatrix}
    0\\
    e_i
\end{pmatrix}.
\end{array}
\]
For $t$ sufficiently small we have
$Q_0\left(x^t\right) \subset Q_0(\bar x)$ and 
$\mathcal{E}\left(y^t\right) \subset \mathcal{E}\left(\bar y\right)$. In addition, it holds $\mathcal{H}\left(x^t,y^t\right) \subset a_{01}(\bar x,\bar y) \cup a_{10}(\bar x,\bar y) \cup a_{00}(\bar x,\bar y)$ and $\mathcal{N}\left(y^t\right) \subset a_{10}(\bar x,\bar y) \cup a_{00}(\bar x,\bar y)$. By setting some $\mu$-multipliers to be zero if needed, we equivalently obtain:
\[
   \begin{array}{rcl}
  \begin{pmatrix}
  0\\
   0
  \end{pmatrix}&=& \displaystyle \sum\limits_{p \in P}  \lambda^t_p 
  \begin{pmatrix}
  \nabla h_p( x^t)\\
  0
  \end{pmatrix}+
    \sum\limits_{q \in Q_0(\bar x)}  \mu^t_{1,q} 
    \begin{pmatrix}
    \nabla g_q( x^t)\\
    0
    \end{pmatrix}+
    \sum\limits_{i\in\mathcal{E}(\bar y)}   \mu^t_{2,i} \begin{pmatrix}
    0\\
     e_i
\end{pmatrix} \\ 
    && \displaystyle+
 \mu^t_3 \begin{pmatrix}
0\\
e
\end{pmatrix}
+ \sum\limits_{i\in\mathcal{H}( x^t,  y^t)\cap a_{01}\left(\bar x,\bar y\right)}   \eta^t_{i} \begin{pmatrix}
    y_{i}^te_{ i}\\
    x_{i}^te_{ i}
\end{pmatrix}
+ \sum\limits_{{i}\in\mathcal{H}( x^t,  y^t)\cap a_{10}\left(\bar x,\bar y\right)}   \eta^t_{i} \begin{pmatrix}
    y_{i}^te_{ i}\\
    x_{i}^te_{ i}
\end{pmatrix}
\\ 
&&+\displaystyle \sum\limits_{i\in\mathcal{H}( x^t,  y^t)\cap a_{00}\left(\bar x,\bar y\right)}   \eta^t_{i} \begin{pmatrix}
     y^t_{i}e_{i}\\
   x^t_{i}e_{i}
\end{pmatrix}\\ 
&&+
    \displaystyle\sum\limits_{i\in\mathcal{N}( y^t)\cap a_{10}\left(\bar x,\bar y\right)}   \nu^t_{i} \begin{pmatrix}
    0\\
    e_{i}
\end{pmatrix}
+
    \sum\limits_{i\in\mathcal{N}( y^t)\cap a_{00}\left(\bar x,\bar y\right)}   \nu^t_{i} \begin{pmatrix}
    0\\
    e_{i}
\end{pmatrix}.
\end{array}
\]
This, however, implies that not all multipliers in the following equation vanish:
\[
   \begin{array}{rcl}
  \begin{pmatrix}
  0\\
   0
  \end{pmatrix}&=& \displaystyle \sum\limits_{p \in P} \hat \lambda^t_p 
  \begin{pmatrix}
  \nabla h_p(x^t)\\
  0
  \end{pmatrix}+
    \sum\limits_{q \in Q_0(\bar x)} \hat \mu^t_{1,q} 
    \begin{pmatrix}
    \nabla g_q( x^t)\\
    0
    \end{pmatrix}+
    \sum\limits_{i\in\mathcal{E}(\bar y)}  \hat \mu^t_{2,i} \begin{pmatrix}
    0\\
     e_i
\end{pmatrix} \\
    && \displaystyle+
 \hat \mu^t_3 \begin{pmatrix}
0\\
e
\end{pmatrix}
+ \sum\limits_{i\in\mathcal{H}( x^t,  y^t)\cap a_{01}\left(\bar x,\bar y\right)}  \hat \eta^t_{i} \begin{pmatrix}
    y_{i}^te_{ i}\\
    x_{i}^te_{ i}
\end{pmatrix}
+ \sum\limits_{{i}\in\mathcal{H}( x^t,  y^t)\cap a_{10}\left(\bar x,\bar y\right)}  \hat \eta^t_{i} \begin{pmatrix}
    y_{i}^te_{ i}\\
    x_{i}^te_{ i}
\end{pmatrix}
\\ 
&&+\displaystyle\sum\limits_{i\in\mathcal{H}( x^t,  y^t)\cap a_{00}\left(\bar x,\bar y\right)}  \hat \eta^t_{1,i} \begin{pmatrix}
     e_{i}\\
   0
\end{pmatrix}
+\sum\limits_{i\in\mathcal{H}( x^t,  y^t)\cap a_{00}\left(\bar x,\bar y\right)}  \hat \eta^t_{2,i} \begin{pmatrix}
     0\\
   e_{i}
\end{pmatrix}
\\ 
&&+
   \displaystyle \sum\limits_{i\in\mathcal{N}( y^t)\cap a_{10}\left(\bar x,\bar y\right)}  \hat\nu^t_{i} \begin{pmatrix}
 y_{i}^te_{ i}\\
    x_{i}^te_{ i}
\end{pmatrix}
+
    \sum\limits_{i\in\mathcal{N}( y^t)\cap a_{00}\left(\bar x,\bar y\right)}  \hat\nu^t_{i} \begin{pmatrix}
    0\\
    e_{i}
\end{pmatrix}.
\end{array}
\]
A contradiction to (\ref{eq:MPOCLICQxt}) follows by taking into account that $ \mathcal{H}( x^t,  y^t) \cap \mathcal{N}( y^t) = \emptyset$.
If instead we suppose that there is no subsequence with $\sum\limits_{i=1}^{n} y^t_i = n - s$, then we can consider a subsequence with  $\sum\limits_{i=1}^{n} y^t_i > n - s$. By following a similar argumentation, we produce a contradiction to (\ref{eq:MPOCLICQxt}) again. \end{proof}

Next, we give the definitions of a (nondegenerate) Karush-Kuhn-Tucker point of $\mathcal{S}$ and of its quadratic index as it is meanwhile standard in nonlinear programming, see e.g. \cite{Jongen:2000}.
 
\begin{definition}[Karush-Kuhn-Tucker point]
    A feasible point $(x, y)$ of $\mathcal{S}$ is called Kurush-Kuhn-Tucker if there exist multipliers 
\[
\lambda_p, p \in P,
 \mu_{1,q}, q\in Q_0( x), \mu_{2,i},i \in \mathcal{E}( y), \mu_3,
\eta^\ge_i, \eta^\le_i, \nu_i, i\in\left\{1,\ldots,n\right\},
\]
such that the following conditions hold:
\begin{equation}
   \label{eq:kkt-1} 
   \begin{array}{rcl}
  \begin{pmatrix}
   \nabla f( x)\\
   c
  \end{pmatrix}&=& \displaystyle \sum\limits_{p \in P} \lambda_p 
  \begin{pmatrix}
  \nabla h_p( x)\\
  0
  \end{pmatrix}+
    \sum\limits_{q \in Q_0( x)} \mu_{1,q} 
    \begin{pmatrix}
    \nabla g_q( x)\\
    0
    \end{pmatrix}
   - \sum\limits_{i\in\mathcal{E}( y)}  \mu_{2,i} \begin{pmatrix}
    0\\
     e_i
\end{pmatrix}+
 \mu_3 \begin{pmatrix}
0\\
e
\end{pmatrix} \\ 
    && \displaystyle
+ \sum\limits_{i\in\mathcal{H}^\ge( x,  y)}  \eta^\ge_{i} \begin{pmatrix}
    y_ie_i\\
   x_ie_i
\end{pmatrix}
-\sum\limits_{i\in\mathcal{H}^\le( x,  y)}  \eta^\le_{i} \begin{pmatrix}
    y_ie_i\\
    x_ie_i
\end{pmatrix}
+
    \sum\limits_{i\in\mathcal{N}( y)}  \nu_{i} \begin{pmatrix}
    0\\
    e_i
\end{pmatrix},
\end{array}
\end{equation}
\begin{equation}
   \label{eq:kkt-2}  \mu_{1,q} \ge 0,q\in Q_0\left( x\right),
    \mu_{2,i} \ge 0, i\in \mathcal{E}(y),
    \mu_3\ge 0, \mu_3\left(
   \sum\limits_{i=1}^{n}  y_i -(n-s)\right)=0,
\end{equation}
\begin{equation}
   \label{eq:kkt-3} 
  \eta^\ge_i\ge 0,i \in \mathcal{H}^\ge( x,  y), \eta^\le_i\ge 0,i \in \mathcal{H}^\le( x,  y), \nu_i\ge 0, i \in \mathcal{N}( y).
\end{equation}
\end{definition}
We again define the Lagrange function as

\[
        \begin{array}{rcl}
        L^\mathcal{S}(x,y)&=& \displaystyle
   f(x)+c^Ty
       - \sum\limits_{p \in P} \lambda_p 
  h_p(x)-
    \sum\limits_{q \in Q_0(x)} \mu_{1,q} 
    g_q(x)
      \\&& \displaystyle +\sum\limits_{i\in\mathcal{E}(y)}
   \mu_{2,i}
    \left(y_i - (1+\varepsilon)\right)
 -\mu_3\left(\sum\limits_{i=1}^n y_i - (n-s)\right)  \\ 
    && \displaystyle
- \sum\limits_{i\in\mathcal{H}^\ge( x,  y)}  \eta^\ge_{i}
    \left(x_iy_i +t \right)
+\sum\limits_{i\in\mathcal{H}^\le( x,  y)}  \eta^\le_{i}
    \left(x_iy_i-t\right)
-
    \sum\limits_{i\in\mathcal{N}( y)}  \nu_{i}
    y_i.
\end{array}
\]
The tangent space is given by

\[
    \mathcal{T}^\mathcal{S}_{(x,y)} =\left\{
\xi \in \R^{2n}\,\left\vert\, \begin{array}{l} \begin{pmatrix}
Dh_p( x),0\end{pmatrix} \xi=0, p \in P, 
\begin{pmatrix}
Dg_q( x),0\end{pmatrix}\xi=0,q \in Q_0( x),\\ \displaystyle
\begin{pmatrix}
0,e_i
\end{pmatrix}\xi=0, i\in \mathcal{E}(y),
\begin{pmatrix}
0,e
\end{pmatrix}\xi=0 \mbox{ if } \sum_{i=1}^{n} y_i = n - s,\\
\begin{pmatrix}
y_ie_i,x_ie_i
\end{pmatrix}\xi=0, i \in \mathcal{H}\left(x,y\right), 
\begin{pmatrix}
    0,e_i
\end{pmatrix}\xi=0, i \in \mathcal{N}(y)
\end{array}
\right.\right\}.
\]

\begin{definition}[Nondegenerate Karush-Kuhn-Tucker point]
A Karush-Kuhn-Tucker point $(x, y)$ of $\mathcal{S}$ with multipliers $(\lambda, \mu, \eta, \nu)$
is called nondegenerate if 

ND1: LICQ holds at $( x, y)$,

ND2: $ \mu_{1,q}>0$, $q \in Q_0\left(x\right)$, $ \mu_{2,i}>0$, $i \in \mathcal{E}\left(y\right)$,
    $\eta^\ge_i>0$, $i \in \mathcal{H}^\ge(x,y)$,
    $\eta^\le_i>0$, $i \in \mathcal{H}^\le(x,y)$,
    $\nu_i >0$, $i \in \mathcal{N}(y)$, and $ \mu_3>0$ if  $\sum\limits_{i=1}^{n} y_i =n-s$,

ND3: the matrix $D^2 L^\mathcal{S}(x,y) \restriction_{\mathcal{T}^\mathcal{S}_{(x,y)}}$ is nonsingular.
\end{definition}

\begin{definition}[Quadratic index]
Let $(x,y)$ be a Karush-Kuhn-Tucker point of $\mathcal{S}$ with unique multipliers $(\lambda, \mu, \eta, \nu)$.
The  number of negative eigenvalues of the matrix $D^2 L^\mathcal{S}(x,y)\restriction_{\mathcal{T}^\mathcal{S}_{(x,y)}}$ is called its quadratic index ($QI$).
\end{definition}

Note that ND1-ND3 are usual assumptions in nonlinear programming. ND1 refers to the linear independence constraint qualification, ND2 means the strict complementarity, and ND3 describes the second-order regularity. For the index of a nondegenerate Karush-Kuhn-Tucker point just the quadratic part is essential.

Next Lemma \ref{lem:EandH} examines the structure of $y$-components of a Karush-Kuhn-Tucker point of $\mathcal{S}$.

\begin{lemma}[Auxiliary $y$-variables in $\mathcal{S}$]
\label{lem:EandH}
Let $(x, y)$ be a Karush-Kuhn-Tucker point of $\mathcal{S}$. Then, it holds:

a) the summation inequality constraint is active, i.\,e. $\sum\limits_{i=1}^n  y_i=n-s$,

b) the index set $\mathcal{E}( y) \cup \mathcal{H}(x, y)$ consists of at least $n-s-1$ elements,
and the index set $\mathcal{N}( y)$ consists of at most $s$ elements. Additionally, there is at most one index, that does not belong to any of these sets, i.\,e. $\left\vert \mathcal{O}\left( x,  y\right)\right\vert \le 1$.
\end{lemma}
\begin{proof}
a) Let $(x, y)$ be a Karush-Kuhn-Tucker point of $\mathcal{S}$ and $\sum\limits_{i=1}^n  y_i>n-s$. Then, there exist multipliers $(\lambda,  \mu,  \eta,  \nu)$, such that (\ref{eq:kkt-1})--(\ref{eq:kkt-3}) are fulfilled. Since $ \mu_3=0$, we have that the $(n+i)$-th row of (\ref{eq:kkt-1}) reads as
\[
c_i=\left\{
\begin{array}{ll}
     - \mu_{2,i},&\mbox{for }i \in \mathcal{E}( y)\backslash\mathcal{H}( x, y),  \\
    - \mu_{2,i}+ \eta^\ge_i  x_i, &\mbox{for }i \in \mathcal{E}( y) \cap \mathcal{H}^\ge( x, y),\\
    - \mu_{2,i}- \eta^\le_i  x_i, &\mbox{for }i \in \mathcal{E}( y) \cap  \mathcal{H}^\le( x, y),\\
 \eta^\ge_i  x_i, &\mbox{for }i \in \mathcal{H}^\ge( x, y)\backslash \mathcal{E}( y),\\
- \eta^\le_i  x_i, &\mbox{for }i \in \mathcal{H}^\le( x, y)\backslash \mathcal{E}( y),\\
 \nu_i,&\mbox{for }i \in \mathcal{N}( y),\\
     0,&\mbox{else}.
\end{array}\right.
\]
Due to (\ref{eq:kkt-2}), (\ref{eq:kkt-3}), and $c>0$, it must hold that $i\in \mathcal{N}( y)$ for all $i \in \left\{1,\ldots,n\right\}$. This, however, contradicts $\sum\limits_{i=1}^n  y_i>n-s$.

b) 
As in the proof of statement a), we conclude that $ \mu_3>0$ for a Karush-Kuhn-Tucker point $( x, y)$ of $\mathcal{S}$. Hence, the $(n+i)$-th row now reads as
\begin{equation}
\label{eq:kth-rowkkt}    
c_i=\left\{
\begin{array}{ll}
     - \mu_{2,i}+ \mu_3,&\mbox{for }i \in \mathcal{E}( y)\backslash \mathcal{H}( x, y),  \\
    - \mu_{2,i}+ \mu_3+ \eta^\ge_i  x_i, &\mbox{for }i \in \mathcal{E}( y) \cap \mathcal{H}^\ge( x, y),\\
    - \mu_{2,i}+ \mu_3- \eta^\le_i  x_i, &\mbox{for }i \in \mathcal{E}( y) \cap  \mathcal{H}^\le( x, y),\\
 \mu_3+ \eta^\ge_i  x_i, &\mbox{for }i \in \mathcal{H}^\ge( x, y)\backslash \mathcal{E}( y),\\
 \mu_3- \eta^\le_i  x_i, &\mbox{for }i \in \mathcal{H}^\le( x, y)\backslash \mathcal{E}( y),\\
 \mu_3+ \nu_i,&\mbox{for }i \in \mathcal{N}( y),\\
      \mu_3,&\mbox{else}.
\end{array}\right.
\end{equation}
It follows from (\ref{eq:kth-rowkkt}) and  the components of $c$ being pairwise different that there can be at most one element ${\bar i} \in \mathcal{O}\left( x,  y\right)$. 
If $\mathcal{E}( y) \cup \mathcal{H}( x,  y)$ consists of fewer than $n-s-1$ elements, we get:
\[
\sum\limits_{i=1}^n  y_i\le (n-s-2)\cdot (1+\varepsilon)+y_{\bar i}<(n-s-1)\cdot(1+\varepsilon)<n-s,
\]
a contradiction.
Finally, we assume that $\mathcal{N}( y)$ consists of more than $s$ elements.
In this case, there are at most $n-s-1$ nonvanishing components of $y$. Consequently, 
\[
\sum\limits_{i=1}^n  y_i\le (n-s-1)\cdot (1+\varepsilon)<n-s
\]
provides a contradiction.
\end{proof}

We apply the general result on the Scholtes-type regularization of MPOC in our context for the regularized continuous reformulation $\mathcal{R}$, see \cite{laemmel:mpoc}.

\begin{theorem}[Convergence from $\mathcal{S}$ to $\mathcal{R}$, cf. \cite{laemmel:mpoc}]
\label{thm:regul}
Suppose that a sequence of Karush-Kuhn-Tucker points $(x^{t},y^{t})$ of 
$\mathcal{S}$ converges to $\left(\bar x,\bar y\right)$ for $t \rightarrow 0$. If MPOC-LICQ holds at $\left(\bar x,\bar y\right)$, then it is a T-stationary point of $\mathcal{R}$.
\end{theorem}

From the proof of Theorem \ref{thm:regul} in \cite{laemmel:mpoc} also the convergence of the corresponding multipliers can be deduced.

\begin{remark}[Convergence of multipliers]
\label{rem:multipliers}
 Let $\left(\lambda^t,\mu^t,\eta^t,\nu^t\right)$ be the multipliers of the Karush-Kuhn-Tucker points $\left(x^{t},y^{t}\right)$ of $\mathcal{S}$ and $\left(\bar \lambda,\bar \mu,\bar \sigma, \bar \varrho\right)$ of the T-stationary point $\left(\bar x, \bar y\right)$ of $\mathcal{R}$ as in Theorem \ref{thm:regul}. Due to MPOC-LICQ at $\left(\bar x, \bar y\right)$, we have: 
 
a) $\lim\limits_{t \to 0} \lambda^t=\bar \lambda$,
    $\lim\limits_{t \to 0} \mu^t=\bar \mu$,

b) $\lim\limits_{t \to 0} \left(\eta_{i}^{\ge,t}-\eta_{i}^{\le,t}\right)y_{i}^t=\bar \sigma_{1,i}$, $i \in a_{01}\left(\bar x, \bar y\right)$,
 
 c) $\lim\limits_{t \to 0} \nu_{i}^{t}+ \left(\eta_{i}^{\ge,t}-\eta_{i}^{\le,t}\right)x_{i}^t=\bar \sigma_{2,i}$, $i \in a_{10}\left(\bar x, \bar y\right)$,

d) $\lim\limits_{t \to 0} \left(\eta_{i}^{\ge,t}-\eta_{i}^{\le,t}\right)y_{i}^t=\bar \varrho_{1,i}$, $\lim\limits_{t \to 0} \nu_{i}^{t}+ \left(\eta_{i}^{\ge,t}-\eta_{i}^{\le,t}\right)x_{i}^t=\bar \varrho_{2,i}$, $i \in a_{00}\left(\bar x, \bar y\right)$.
\end{remark}

The convergence of nondegenerate Karush-Kuhn-Tucker points of $\mathcal{S}$ does not prevent the limiting T-stationary point of $\mathcal{S}$ from being degenerate. Let us present in Example \ref{ex:ndt2} the failure of NDT2. Examples with the failure of NDT1, NDT3, or NDT4 are not difficult to construct analogously. 

\begin{example}[Failure of NDT2]
\label{ex:ndt2}
We consider the following Scholtes-type regularization $\mathcal{S}$ with $n=2$ and $s=1$:
\[
\begin{array}{rl}
\mathcal{S}: \quad \min\limits_{x,y}& (x_1-1)^2+(x_2-1)^2+c_1y_1+(c_1+\frac{5}{36})y_2\\
\mbox{s.t.}&1+x_1-x_2\ge 0, \\ &y_1+y_2\ge 1, \quad  
-t\le x_i y_i \le t,  \quad 0\leq y_i\le 1+\varepsilon, \quad i=1,2,
\end{array}
\]
as well as the point $(x^t,y^t)=(t,1,1,0)$.We claim that this point is a nondegenerate Karush-Kuhn-Tucker point for $t<\frac{1}{2}-\sqrt{\frac{13}{72}}$.
Indeed, it holds:
\[
\begin{pmatrix}
2t-2\\0\\c_1\\c_1+\frac{5}{36}
\end{pmatrix}
=
\mu_3^t
\begin{pmatrix}
0\\0\\1\\1
\end{pmatrix}
-\eta_1^{\le,t}
\begin{pmatrix}
1\\
0\\
t\\
0
\end{pmatrix}
+\nu_2^t
\begin{pmatrix}
0\\0\\0\\1
\end{pmatrix}
\]
with the positive multipliers 
$\mu_3^t=c_1+2t-2t^2$, $\eta_1^{\le,t}=2-2t$, $\nu_2^t=\frac{5}{36}-2t+2t^2$.
The tangent space  is 
$
 \mathcal{T}^\mathcal{S}_{(x^t,y^t)} =\left\{
\xi \in \R^{4}\,\left\vert\,\xi_1=\xi_3=\xi_4=0 \right.\right\}$.
The Hessian of the corresponding Lagrange function is 
 \[
 D^2 L^\mathcal{S}(x^t,y^t)=\begin{pmatrix}
 2&0&2-2t&0\\
 0&2&0&0\\
 2-2t&0&0&0\\
 0&0&0&0
 \end{pmatrix}.
 \]
Therefore, it is straightforward to see that 
$D^2 L^\mathcal{S}(x^t,y^t)\restriction_{\mathcal{T}^{\mathcal{S}}_{(x^t,y^t)}}$ is nonsingular.
We conclude that ND1-ND3 are fulfilled at the $(x^t,y^t)$. Moreover, $(x^t,y^t)$ converges to $(\bar x, \bar y)=(0,1,1,0)$ if $t \to 0$.
This point is T-stationary for the corresponding regularized continuous reformulation $\mathcal{R}$ according to Theorem \ref{thm:regul}, since  MPOC-LICQ is fulfilled. Indeed, we obtain the T-stationarity condition
\[
\begin{pmatrix}
-2\\0\\c_1\\c_1+\frac{5}{36}
\end{pmatrix}
=
\bar \mu_{1}\begin{pmatrix}
1\\-1\\0\\0
\end{pmatrix}
+
\bar \mu_3
\begin{pmatrix}
0\\0\\1\\1
\end{pmatrix}
+
\bar \sigma_{1,1}
\begin{pmatrix}
1\\0\\0\\0
\end{pmatrix}
+
\bar \sigma_{2,2}
\begin{pmatrix}
0\\0\\0\\1
\end{pmatrix}
\]
with the unique multipliers 
$\bar \mu_{1}=0, \bar \mu_3=c_1, \bar \sigma_{1,1}=-2, \bar \sigma_{2,2}=\frac{5}{36}$.
However, NDT2 is violated at $(\bar x, \bar y)$.
\qed
\end{example}

Due to Example \ref{ex:ndt2}, we cannot expect that a T-stationary point of $\mathcal{R}$, which is the limit of a sequence of nondegenerate Karush-Kuhn-Tucker points of $\mathcal{S}$, is also nondegenerate. Instead, we intend to examine its type if assuming nondegeneracy.
Next Lemma \ref{lem:QandE} provides some valuable insights into the relations between active index sets while doing so.

\begin{lemma}[Active index sets]
\label{lem:QandE}
Suppose a sequence of Karush-Kuhn-Tucker points $\left(x^{t},y^{t}\right)$ of $\mathcal{S}(t)$ converges to $\left(\bar x,\bar y\right)$ for $t \to 0$. Moreover, let $(\bar x,\bar y)$ be a nondegenerate T-stationary point of $\mathcal{R}$. Then, for all sufficiently small $t$ it holds:

a) $Q_0\left(\bar x\right)=Q_0\left(x^{t}\right)$,

b) $\mathcal{E}\left(\bar y\right)=\mathcal{E}\left(y^{t}\right) $,

c) $a_{00}\left(\bar x, \bar y\right) \subset \mathcal{H}\left(x^t,y^t\right)$,

d) $\mathcal{N}\left(y^t\right)\subset
a_{10}\left(\bar x,\bar y\right)\subset 
\mathcal{N}\left(y^t\right) \cup \mathcal{H}\left(x^t,y^t\right)$.    

\end{lemma}
\begin{proof}
a) We start by proving $Q_0\left(\bar x\right)=Q_0\left(x^{t}\right)$.
Due to continuity arguments, we have $Q_0\left(x^{t}\right)\subset Q_0\left(\bar x\right)$ for all sufficiently small $t$.
Let us now assume that there exists
$\bar i \in  Q_0\left(\bar x\right)\backslash Q_0\left(x^{t}\right)$ along a subsequence.
Hence, for the corresponding multipliers it holds $\mu_{\bar i}^t=0$.
NDT1 allows us to apply Remark \ref{rem:multipliers}, and we thus have $\bar \mu_{\bar i}=\lim\limits_{t \rightarrow \infty} \mu_{\bar i}^t=0$, a contradiction to NDT2. Consequently, $Q_0\left(\bar x\right)=Q_0\left(x^{t}\right)$ holds for all sufficiently small $t$.

b) Next, we prove $\mathcal{E}\left(\bar y\right)=\mathcal{E}\left(y^{t}\right)$.
Again, continuity arguments provide $\mathcal{E}\left(y^{t}\right)\subset \mathcal{E}\left(\bar y\right)$ for all sufficiently small $t$.
Similar to the first part of the proof, we now assume there exists $\bar i \in \mathcal{E}\left(\bar y\right)\backslash \mathcal{E}\left(y^{t}\right)$ along a subsequence. 
As we have seen in Lemma \ref{lem:a01}, T-stationarity of $(\bar x, \bar y)$ implies in particular $c_{\bar i}=-\bar \mu_{2,\bar i}+\bar \mu_{3,\bar i}$. 
Moreover, NDT1 and Remark \ref{rem:multipliers} provide 
$\lim\limits_{t \to 0}\mu_{3}^{t}=\bar \mu_{3}$.
Since $\bar i \notin\mathcal{N}\left(y^{t}\right)$, we distinguish the following cases:

$(i)$ $\bar i \in\mathcal{H}^{\ge}\left(x^{t},y^{t}\right)\backslash \mathcal{E}\left(y^{t}\right)$. Karush-Kuhn-Tucker conditions for $\left(x^t, y^t\right)$ imply $c_{\bar i}= \mu^t_{3,\bar i}+ \eta^{\ge,t}_{\bar i}x_{\bar i}$, cf. (\ref{eq:kth-rowkkt}).
It follows $
-\bar \mu_{2,\bar i}+\bar \mu_{3,\bar i}=\mu^t_{3,\bar i}+ \eta^{\ge,t}_{\bar i}x_{\bar i}$.
By taking the limit, we can cancel out $\bar \mu_{3,\bar i}$ and $\mu^t_{3,\bar i}$.
This leads to a contradiction because the left-hand side of the equation is strictly negative due to NDT2 and the right-hand side is nonnegative since  $\eta^{\ge,t}_{\bar i}$ is nonnegative and $x_{\bar i}$ is positive.

$(ii)$ $\bar i \in\mathcal{H}^{\le}\left(x^{t},y^{t}\right)\backslash \mathcal{E}\left(y^{t}\right)$.
By using (\ref{eq:kth-rowkkt}), we get
$c_{\bar i}= \mu^t_{3,\bar i}- \eta^{\le,t}_{\bar i}x_{\bar i}$.
This leads to a contradiction just as in the previous case.

$(iii)$ $\bar i \in\mathcal{O}\left(x^{t},y^{t}\right)$.
Analogously, we obtain $c_{\bar i}= \mu^t_{3,\bar i}$ from (\ref{eq:kth-rowkkt}).
It follows
$-\bar \mu_{2,\bar i}+\bar \mu_{3,\bar i}=\mu^t_{3,\bar i}$. 
Taking the limits leads to $\bar \mu_{2,\bar i}=0$, a contradiction with NDT2.

\noindent
 Altogether, $\mathcal{E}\left(\bar y\right)\backslash \mathcal{E}\left(y^{t}\right)=\emptyset$ for all sufficiently small $t$, and the assertion follows. 

 c) Clearly, $a_{00}\left(\bar x, \bar y\right) \cap \mathcal{E}(y^t)=\emptyset$ for sufficiently small $t$.
Let us assume  there exists an $\bar i \in a_{00}(\bar x,\bar y)\cap\mathcal{N}\left(y^t\right)$. Due to (\ref{eq:kth-rowkkt}), we then have $c_{\bar i}= \mu_3^t + \nu_{\bar i}^t$,
whereas (\ref{eq:k-throw}) provides
$ c_{\bar i}=\bar \mu_3 + \bar \varrho_{2,\bar i}$.
According to Remark \ref{rem:multipliers}, we have 
$\lim\limits_{t \to 0}\mu_{3}^{t}=\bar \mu_3$. Consequently, it must hold
$
\lim\limits_{t \to 0} \nu_{\bar i}^t=\bar \varrho_{2,\bar i}$.
This, however, cannot be true since $\nu_{\bar i}^t\geq 0$, while
$\bar \varrho_{2,\bar i}<0$ due to NDT3 from the nondegeneracy of $(\bar x, \bar y)$, a contradiction. 
Let us assume now that there exists an $\bar i \in a_{00}(\bar x,\bar y)\cap \mathcal{O}\left(x^t,y^t\right)$. Analogously, we get $\bar \varrho_{2,\bar i}=0$, again a contradiction to NDT3. Overall, we get the assertion.

 d) Clearly, $a_{01}\left(\bar x,\bar y\right)\cap \mathcal{N}(y^t)  =\emptyset$ for sufficiently small $t$. From c) we also know that $a_{00}\left(\bar x,\bar y\right)\cap \mathcal{N}(y^t)  =\emptyset$. Altogether, the first inclusion of the assertion follows immediately. Further, it also holds
$a_{10}\left(\bar x, \bar y\right) \cap \mathcal{E}(y^t)= \emptyset$ for sufficiently small $t$.
Let us assume there exists an  $\bar i \in a_{10}(\bar x,\bar y)\cap\mathcal{O}\left(x^t,y^t\right)$. Due to (\ref{eq:kth-rowkkt}), we have $c_{\bar i}=\mu_3^t$.
In view of Lemma \ref{lem:a01}c), there exists an index $\tilde i \in a_{01}(\bar x,\bar y)\backslash \mathcal{E}\left(\bar y\right)$. Thus, T-stationarity of $(\bar x, \bar y)$ implies via (\ref{eq:k-throw}) that $c_{\tilde i}=\bar \mu_3$.
By taking the limit and Remark \ref{rem:multipliers}, we obtain $c_{\bar i}=c_{\tilde i}$, but $\bar i\not=\tilde i$, a contradiction to the choice of $c$.
\end{proof}

Next Theorem \ref{thm:kktsequence} highlights the convergence properties of the Scholtes-type regularization method.

\begin{theorem}[Convergence from $\mathcal{S}$ to $\mathcal{R}$ again]
\label{thm:kktsequence}
Suppose that a sequence of nondegenerate Karush-Kuhn-Tucker points $(x^{t},y^{t})$ of 
$\mathcal{S}$ with quadratic index $m$ converges to $\left(\bar x,\bar y\right)$ for $t \to 0$. If $(\bar x,\bar y)$ is a nondegenerate T-stationary point of $\mathcal{R}$, then we have for its T-Index: 
\[
\max\left\{m - \left\vert\left\{i\in a_{01}\left(\bar x,\bar y\right)\,\left\vert\,\bar \sigma_{1,i}=0\right.\right\} \right\vert, 0\right\}\le TI \leq m.
\]
If additionally NDT6 holds at $(\bar x,\bar y)$, then the indices coincide, i.e. $TI= m$.
\end{theorem}
\begin{proof}
The proof will be divided into 4 major steps.

{\bf Step 1a.}
We rewrite the tangent space corresponding to the Karush-Kuhn-Tucker point $\left(x^t,y^t\right)$. For that, we use Lemma \ref{lem:EandH}a) which provides that the summation constraint is active: 
\[
    \mathcal{T}^\mathcal{S}_{\left(x^t,y^t\right)}=\left\{
\xi \in \R^{2n}\,\left\vert\, \begin{array}{l} \begin{pmatrix}
Dh_p( x^t),0\end{pmatrix} \xi=0, p \in P, 
\begin{pmatrix}
Dg_q( x^t),0\end{pmatrix}\xi=0,q \in Q_0( x^t),\\
\begin{pmatrix}
0,e
\end{pmatrix}\xi=0,
\begin{pmatrix}
0,e_i
\end{pmatrix}\xi=0, i\in \mathcal{E}(y^t),\\
\begin{pmatrix}
y_i^te_i,x_i^te_i
\end{pmatrix}\xi=0, i \in \mathcal{H}\left(x^t,y^t\right)\cap a_{00}\left(\bar x, \bar y\right),\\
\begin{pmatrix}
y_i^te_i,x_i^te_i
\end{pmatrix}\xi=0, i \in \mathcal{H}\left(x^t,y^t\right)\cap a_{01}\left(\bar x, \bar y\right),\\
\begin{pmatrix}
y_i^te_i,x_i^te_i
\end{pmatrix}\xi=0, i \in \mathcal{H}\left(x^t,y^t\right)\cap a_{10}\left(\bar x, \bar y\right),\\
\begin{pmatrix}
    0,e_i
\end{pmatrix}\xi=0, i \in \mathcal{N}\left(y^t\right)
\end{array}
\right.\right\}.
\]
In total there are, due to LICQ,
\[
\begin{array}{rcl}
\alpha^\mathcal{S}_t&=&
\left\vert P\right\vert +\left\vert Q_0\left(x^t\right)\right\vert +1+
\left\vert \mathcal{E}\left(y^t\right)\right\vert +
\left\vert \mathcal{H}\left(x^t,y^t\right)\cap a_{00}\left(\bar x,\bar y\right)\right\vert 
\\
&&
+\left\vert \mathcal{H}\left(x^t,y^t\right)\cap a_{01}\left(\bar x,\bar y\right)\right\vert 
+\left\vert \mathcal{H}\left(x^t,y^t\right)\cap a_{10}\left(\bar x,\bar y\right)\right\vert 
+\left\vert \mathcal{N}\left(y^t\right)\right\vert 
\end{array}
\]
linearly independent vectors involved. We use Lemma \ref{lem:QandE}a) and \ref{lem:QandE}b) to
substitute $\left\vert Q_0\left(x^t\right)\right\vert $ with $\left\vert Q_0\left(\bar x\right)\right\vert $ and $\left\vert \mathcal{E}\left(y^t\right)\right\vert $ with
$\left\vert \mathcal{E}\left(\bar y\right)\right\vert $, respectively. The latter set has cardinality of  $n-s-1$ due to Lemma \ref{lem:a01}c). Additionally, we use Lemma \ref{lem:QandE}c) and \ref{lem:QandE}d) to conclude:
\[
\begin{array}{rcl}
\alpha^\mathcal{S}_{t}&=&
\left\vert P\right\vert +\left\vert Q_0\left(\bar x\right)\right\vert +1+
n-s-1
+\left\vert a_{00}\left(\bar x,\bar y\right)\right\vert 
\\
&&
+\left\vert \mathcal{H}\left(x^t,y^t\right)\cap a_{01}\left(\bar x,\bar y\right)\right\vert 
+\left\vert a_{10}\left(\bar x,\bar y\right)\right\vert .
\end{array}
\]
Finally, $\left\vert a_{00}\left(\bar x,\bar y\right)\right\vert 
+\left\vert a_{10}\left(\bar x,\bar y\right)\right\vert =s$, cf. Lemma \ref{lem:a01}b).
Thus, we have:
\[
\begin{array}{rcl}
\alpha^\mathcal{S}_{t}&=&
\left\vert P\right\vert +\left\vert Q_0\left(\bar x\right)\right\vert 
+\left\vert \mathcal{H}\left(x^t,y^t\right)\cap a_{01}\left(\bar x,\bar y\right)\right\vert 
+n.
\end{array}
\]
{\bf Step 1b.} We examine the tangent space corresponding to the T-stationary point $(\bar x,\bar y)$. For this purpose, we consider the following vectors from its definition:
\[
\begin{pmatrix}
0\\
e_i
\end{pmatrix},i\in \mathcal{E}(\bar y),
\begin{pmatrix}
0\\
e_i
\end{pmatrix}, i \in a_{10}\left(\bar x, \bar y\right)\cup a_{00}\left(\bar x, \bar y\right),
\begin{pmatrix}
0\\
e
\end{pmatrix}.\]
The latter vector is involved due to Lemma \ref{lem:a01}a). The number of these vectors is due to Lemma \ref{lem:a01}c) equal to $(n-s-1)+
s+1=n$. Moreover, they are linearly independent due to MPOC-LICQ.
Hence, we can write the respective tangent space as follows:
\[
    \mathcal{T}^\mathcal{R}_{(\bar x,\bar y)}=\left\{
\xi \in \R^{2n}\,\left\vert\, \begin{array}{l} \begin{pmatrix}

Dh_p(\bar x), 0\end{pmatrix} \xi=0, p \in P, 
\begin{pmatrix}
Dg_q(\bar x), 0\end{pmatrix}\xi=0,q \in Q_0(\bar x),\\
\begin{pmatrix}
e_i,0
\end{pmatrix}\xi=0, i \in a_{00}(\bar x,\bar y) \cup a_{01}(\bar x,\bar y),
\xi_{n+1}=\ldots=\xi_{2n}=0.
\end{array}
\right.\right\}
\]
In total there are, due to MPOC-LICQ,
\[
\alpha^\mathcal{R}=\left\vert P\right\vert +\left\vert Q_0\left(\bar x\right)\right\vert 
+\left\vert a_{00}(\bar x,\bar y)\right\vert +\left\vert a_{01}(\bar x,\bar y)\right\vert +n
\]
linearly independent vectors involved.

{\bf Step 2.}
Let $\mathcal{T}\subset \R^{2n}$ be a linear subspace. We denote the number of negative
eigenvalues of $D^2 L^\mathcal{S}\left( x^t, y^t\right)\restriction_{\mathcal{T}}$ by $QI^\mathcal{S}_{t,\mathcal{T}}$. Analogously, $QI^\mathcal{R}_{t,\mathcal{T}}$ stands for the number of negative eigenvalues of $D^2 L^\mathcal{R}\left( x^t, y^t\right)\restriction_{\mathcal{T}}$ and $\overline{QI}^\mathcal{R}_{\mathcal{T}}$ stands for the number of negative eigenvalues of $D^2 L^\mathcal{R}\left(\bar x , \bar y \right)\restriction_{\mathcal{T}}$. 
We have the following relation between the involved Hessians of the Lagrange functions by denoting $E(i)=e_ie_{n+i}^T+e_{n+i}e_i^T$, $i =1, \ldots, n$:
\begin{equation}
\label{eq:Hessetandbar}
\begin{array}{rcl}
\displaystyle
D^2 L^\mathcal{S}\left( x^t, y^t\right) &=& \displaystyle
D^2 L^\mathcal{R}\left( x^t, y^t\right) - \sum\limits_{i\in \mathcal{H}^\ge\left(x^t,y^t\right)}  \eta_i^{\ge,t} E(i)\displaystyle 
+
\sum\limits_{i\in \mathcal{H}^\le\left(x^t,y^t\right)} \eta_i^{\le,t} E(i).
\end{array}
\end{equation}

{\bf Step 2a.}
It holds for $t$ sufficiently small:
\[
QI^\mathcal{S}_{t,\mathcal{T}^\mathcal{R}_{\left(\bar x,\bar y\right)}}=\overline{QI}^\mathcal{R}_{\mathcal{T}^\mathcal{R}_{\left(\bar x,\bar y\right)}}.
\]
Indeed, by using (\ref{eq:Hessetandbar}), we derive for any $\xi \in  \mathcal{T}^\mathcal{R}_{\left(\bar x,\bar y\right)}$:
\begin{equation}
    \label{eq:s-r-h1}
\begin{array}{rcl}
\xi^T D^2 L^\mathcal{S}\left( x^t, y^t\right) \xi&=&
 \xi^T D^2 L^\mathcal{R}\left(x^t, y^t\right) \xi - \displaystyle \sum\limits_{i\in \mathcal{H}^\ge\left(x^t,y^t\right)} 2\eta_i^{\ge,t} \xi_i  \xi_{n+i} \\
&&\displaystyle +
\sum\limits_{i\in \mathcal{H}^\le\left(x^t,y^t\right)} 2\eta_i^{\le,t}  \xi_i \xi_{n+i}=
 \xi^T D^2 L^\mathcal{R}\left( x^t, y^t\right)  \xi,
\end{array}
\end{equation}
since $\xi_{n+1}=\ldots=\xi_{2n}=0$ as seen in Step 1b.
Hence, we get
 $
QI^\mathcal{S}_{t,\mathcal{T}^\mathcal{R}_{\left(\bar x,\bar y\right)}}=QI^\mathcal{R}_{t,\mathcal{T}^\mathcal{R}_{\left(\bar x,\bar y\right)}}$.
Due to NDT4, continuity arguments provide
$
QI^\mathcal{R}_{t,\mathcal{T}^\mathcal{R}_{\left(\bar x,\bar y\right)}}=\overline{QI}^\mathcal{R}_{\mathcal{T}^\mathcal{R}_{\left(\bar x,\bar y\right)}}$.

{\bf Step 2b.}
We claim that the numbers of positive and negative eigenvalues of $D^2 L^\mathcal{S}\left( x^t, y^t\right)\restriction_{\mathcal{T}^\mathcal{R}_{\left(\bar x,\bar y\right)}}$ and of $D^2 L^\mathcal{S}\left( x^t, y^t\right)\restriction_{\mathcal{T}^\prime}$, respectively, coincide, where
\[
  \mathcal{T}^{\prime} = \left\{
\xi \in \R^{2n}\,\left\vert\, \begin{array}{l} \begin{pmatrix}
Dh_p(x^t),0\end{pmatrix} \xi=0, p \in P, 
\begin{pmatrix}
Dg_q(x^t),0\end{pmatrix}\xi=0,q \in Q_0(\bar x),\\
\begin{pmatrix}
0,e
\end{pmatrix}\xi=0,
\begin{pmatrix}
0,e_i
\end{pmatrix}\xi=0, i\in \mathcal{E}(\bar y),\\
\begin{pmatrix}
e_i,0
\end{pmatrix}\xi=0, i \in a_{00}(\bar x,\bar y),
\begin{pmatrix}
0,e_i
\end{pmatrix}\xi=0, i \in a_{00}(\bar x,\bar y),\\
\begin{pmatrix}
y^t_ie_i,x^t_ie_i
\end{pmatrix}\xi=0, i \in a_{01}(\bar x,\bar y)\cup a_{10}(\bar x,\bar y)
\end{array}
\right.\right\}.
\]
Let $\left\{\lambda^+_1,\ldots,\lambda^+_{k^+}\right\}$ be the positive 
eigenvalues of 
$D^2 L^\mathcal{S}\left( x^t, y^t\right)\restriction_{\mathcal{T}^\mathcal{R}_{\left(\bar x,\bar y\right)}}$ with corresponding eigenvectors 
$\left\{\xi^+_1,\ldots,\xi^+_{k^+}\right\}$.
Hence, for all $k=1,\ldots, k^+$:
\[
{\xi^+_k}^T D^2 L^\mathcal{S}\left( x^t, y^t\right)\xi^+_k>0.
\]
We rewrite the tangent space as follows:
\[
    \mathcal{T}^\mathcal{R}_{\left(\bar x,\bar y\right)}=\left\{
\xi \in \R^{2n}\,\left\vert\, \begin{array}{l} \begin{pmatrix}

Dh_p(\bar x),0\end{pmatrix} \xi=0, p \in P, 
\begin{pmatrix}
Dg_q(\bar x),0\end{pmatrix}\xi=0,q \in Q_0(\bar x),\\
\begin{pmatrix}
0,e
\end{pmatrix}\xi=0,
\begin{pmatrix}
0,e_i
\end{pmatrix}\xi=0, i\in \mathcal{E}(\bar y),\\
\begin{pmatrix}
e_i,0
\end{pmatrix}\xi=0, i \in a_{00}(\bar x,\bar y),
\begin{pmatrix}
0,e_i
\end{pmatrix}\xi=0, i \in a_{00}(\bar x,\bar y),\\
\begin{pmatrix}
\bar y_ie_i,\bar x_ie_i
\end{pmatrix}\xi=0, i \in a_{01}(\bar x,\bar y)\cup a_{10}(\bar x,\bar y)
\end{array}
\right.\right\}.
\]
Due to MPOC-LICQ, the application of the implicit function theorem provides the existence of
$\delta_2,\delta_3>0$ such that for all $k=1,\ldots,k^+$ and $t<\delta_2$ there exists $\xi_{k,t}$ with $\left\| \xi_{k,t}-\xi^+_k\right\|<\delta_3$ and $\xi_{k,t}\in \mathcal{T}^\prime$.
 We can choose $t$ even smaller, such that $\xi_{1,t},\ldots,\xi_{k^+,t}$ remain linearly independent and for all $k=1,\ldots,k^+$ it holds:
\[
{ \xi_{k,t}}^T D^2 L^\mathcal{S}\left( x^t, y^t\right) \xi_{k,t}>0.
\]
%
Hence, $D^2 L^\mathcal{S}\left( x^t, y^t\right)\restriction_{\mathcal{T}^\prime}$ has at least $k^+$ positive eigenvalues. If we repeat the above reasoning for negative eigenvalues, the matrix $D^2 L^\mathcal{S}\left( x^t, y^t\right)\restriction_{\mathcal{T}^\prime}$ has at least as many negative eigenvalues as $D^2 L^\mathcal{S}\left( x^t, y^t\right)\restriction_{\mathcal{T}^\mathcal{R}_{\left(\bar x,\bar y\right)}}$. Additionally, we show that the dimensions of $\mathcal{T}^\mathcal{R}_{\left(\bar x,\bar y\right)}$ and $\mathcal{T}^\prime$ coincide. By Step 1b, we have $2n-\alpha^\mathcal{R}$ for the dimension of $\mathcal{T}^\mathcal{R}_{\left(\bar x,\bar y\right)}$. Since MPOC-LICQ remains valid in the neighborhood of $(\bar x, \bar y)$, we get again 
$2n-\alpha^{\mathcal{R}}$ for the dimension of $\mathcal{T}^\prime$. By continuity arguments, NDT4 and (\ref{eq:s-r-h1}) provide that $D^2 L^\mathcal{S}\left( x^t, y^t\right)\restriction_{\mathcal{T}^\mathcal{R}_{\left(\bar x,\bar y\right)}}$ is nonsingular. Altogether, the assertion follows.

{\bf Step 2c.} 
We claim that 
\[
QI^\mathcal{S}_{t,\mathcal{T}^\mathcal{S}_{\left(x^t,y^t\right)}}\le QI^\mathcal{S}_{t,\mathcal{T}^\mathcal{R}_{\left(\bar x,\bar y\right)}}+\alpha^\mathcal{R}-\alpha^\mathcal{S}_t.
\]
For that, we focus on the dimension of $\mathcal{T}^\mathcal{S}_{\left(x^t,y^t\right)}$. As a consequence of Step 1a it is $2n-\alpha^\mathcal{S}_t$. 
Due to continuity arguments, we can choose $t$ small enough to ensure 
$x^t_{i}\ne 0$, $i\in a_{10}(\bar x,\bar y)$ and
$y^t_{i}\ne 0$, $i\in a_{01}(\bar x,\bar y)$.
Using this and Lemma \ref{lem:QandE}a), \ref{lem:QandE}b), and \ref{lem:QandE}d), it follows
that $\mathcal{T}^\prime \subset \mathcal{T}^\mathcal{S}_{\left(x^t,y^t\right)}$.
Therefore, using Step 2b,  
$
QI^\mathcal{S}_{t,\mathcal{T}^\mathcal{S}_{\left(x^t,y^t\right)}}\le 2n-\alpha^\mathcal{S}_t-k^+$.
We observe in view of NDT4 and Step 1b that 
$
QI^\mathcal{S}_{t,\mathcal{T}^\mathcal{R}_{\left(\bar x,\bar y\right)}}
=2n-\alpha^\mathcal{R}-k^+$.
The assertion follows immediately.

{\bf Step 3.} Let us show that
\[
\max\left\{m - \left\vert \left\{i\in a_{01}\left(\bar x,\bar y\right)\,\left\vert\,\bar \sigma_{1,i}=0\right.\right\} \right\vert , 0\right\}\le TI.
\]
In view of Step 2a, Step 2c and due to continuity, we have for $t$ sufficiently small:

\[
\begin{array}{rcl}
m=QI^\mathcal{S}_{t,\mathcal{T}^\mathcal{S}_{\left(x^t,y^t\right)}}&\overset{Step\,2c}{\le}& QI^\mathcal{S}_{t,\mathcal{T}^\mathcal{R}_{\left(\bar x,\bar y\right)}}+\alpha^\mathcal{R}-\alpha^\mathcal{S}_t \overset{Step\,2a}{=} 
\overline{QI}^\mathcal{R}_{\mathcal{T}^\mathcal{R}_{\left(\bar x,\bar y\right)}}
+\alpha^\mathcal{R}-\alpha^\mathcal{S}_t\\
&\overset{Step\,1}{=}& \overline{QI}^\mathcal{R}_{\mathcal{T}^\mathcal{R}_{\left(\bar x,\bar y\right)}}
+\left\vert a_{00}(\bar x,\bar y)\right\vert +\left\vert a_{01}(\bar x,\bar y)\right\vert -
\left\vert \mathcal{H}\left(x^t,y^t\right)\cap a_{01}\left(\bar x,\bar y\right)\right\vert \\
&=&TI+\left\vert a_{01}(\bar x,\bar y)\right\vert -
\left\vert \mathcal{H}\left(x^t,y^t\right)\cap a_{01}\left(\bar x,\bar y\right)\right\vert .
\end{array}
\]
We show for $t$ sufficiently small:
\[\left\vert a_{01}(\bar x,\bar y)\right\vert -
\left\vert \mathcal{H}\left(x^t,y^t\right)\cap a_{01}\left(\bar x,\bar y\right)\right\vert \le
\left\vert \left\{i\in a_{01}\left(\bar x,\bar y\right)\,\left\vert\,\bar \sigma_{1,i}=0\right.\right\} \right\vert ,
\]
and the assertion will follow immediately since $TI\ge 0$.
Clearly,
\[\left\vert a_{01}(\bar x,\bar y)\right\vert -
\left\vert \mathcal{H}\left(x^t,y^t\right)\cap a_{01}\left(\bar x,\bar y\right)\right\vert =
\left\vert a_{01}(\bar x,\bar y)\backslash \mathcal{H}\left(x^t,y^t\right) \right\vert .
\]
Suppose $\bar i \in a_{01}\left(\bar x,\bar y\right)$ with $\bar \sigma_{1,\bar i}\ne 0$. In view of Remark \ref{rem:multipliers}, the difference $\eta^{\ge,t}_{\bar i}-\eta^{\le,t}_{\bar i}$ cannot vanish for all $t$ sufficiently small. In particular, one of the multipliers $\eta^{\ge,t}_{\bar i}$ or $\eta^{\le,t}_{\bar i}$ has to be not vanishing for all $t$ sufficiently small. Hence, $\bar i\in \mathcal{H}\left(x^t,y^t\right)$. We therefore have: 
\[
a_{01}(\bar x,\bar y)\backslash \mathcal{H}\left(x^t,y^t\right) \subset  \left\{i\in a_{01}\left(\bar x,\bar y\right)\,\left\vert\,\bar \sigma_{1,i}= 0\right.\right\}.
\]

{\bf Step 4.}
Without loss of generality -- considering subsequences if needed -- we can assume that for any $i\in a_{00}(\bar x,\bar y)$ at least one of the sequences $\frac{x_i^t}{y_i^t}$ or $\frac{y_i^t}{x_i^t}$ is convergent. First, we note that the quotients are well defined due to Lemma \ref{lem:QandE}c). Moreover, if the former sequence does not contain a convergent subsequence, we find a subsequence that tends to plus or minus infinity. Consequently, the corresponding subsequence of the latter reciprocal sequence has to converge to zero.
We define the following auxiliary sets:
\[
a_{00}^x(\bar x,\bar y)=\left\{i \in a_{00}(\bar x,\bar y)\,\left\vert\, \frac{x^t_i}{y^t_i}\mbox{ converges for } t \to 0 \right. \right\},
a_{00}^y(\bar x,\bar y)=a_{00}(\bar x,\bar y)\backslash a_{00}^x(\bar x,\bar y).
\]

For $\bar i \in a_{00}^x\left(\bar x,\bar y\right)$ we consider $\mathcal{T}^{\mathcal{R}}_{\left(\bar x,\bar y\right)}$ and replace two of the involved equations, namely 
$\begin{pmatrix}
e_{\bar i},0
\end{pmatrix}\xi=0$ and 
$\begin{pmatrix}
0,e_{\bar i}
\end{pmatrix}\xi=0$ by one equation
$\begin{pmatrix}
e_{\bar i},\lim\limits_{t \to 0}\frac{x_{\bar i}^t}{y_{\bar i}^t}e_{\bar i}
\end{pmatrix}\xi=0$.
Clearly, the vectors involved in the definition of the newly generated linear space, i.\,e.
\[
    \mathcal{T}^{{\bar i}}=\left\{
\xi \in \R^{2n}\,\left\vert\, \begin{array}{l} \begin{pmatrix}

Dh_p(\bar x),0\end{pmatrix} \xi=0, p \in P, 
\begin{pmatrix}
Dg_q(\bar x),0\end{pmatrix}\xi=0,q \in Q_0(\bar x),\\
\begin{pmatrix}
0,e
\end{pmatrix}\xi=0,
\begin{pmatrix}
0,e_i
\end{pmatrix}\xi=0, i\in \mathcal{E}(\bar y),\\
\begin{pmatrix}
e_i,0
\end{pmatrix}\xi=0, i \in a_{00}(\bar x,\bar y)\backslash\{\bar i\},
\begin{pmatrix}
0,e_i
\end{pmatrix}\xi=0, i \in a_{00}(\bar x,\bar y)\backslash\{\bar i\},\\
\begin{pmatrix}
e_{\bar i},\lim\limits_{t \to 0}\frac{x_{\bar i}^t}{y_{\bar i}^t}e_{\bar i}
\end{pmatrix}\xi=0,
\begin{pmatrix}
\bar y_ie_i,\bar x_ie_i
\end{pmatrix}\xi=0, i \in a_{01}(\bar x,\bar y)\cup a_{10}(\bar x,\bar y)\}
\end{array}
\right.\right\},
\]
remain linearly independent.
The dimension of $ \mathcal{T}^{{\bar i}}$ is greater than the dimension of
$\mathcal{T}^{\mathcal{R}}_{\left(\bar x,\bar y\right)}$ by one. Moreover, there exists
$\xi^{\bar i} \in \mathcal{T}^{{\bar i}}$ with 
$
\xi^{\bar i}_{n+\bar i} \not =0$.
Indeed, assume that no such $\xi^{\bar i}$ exists, then we can add the equation 
$\begin{pmatrix}
0,e_{\bar i}
\end{pmatrix}\xi=0$ to the defining equations of $ \mathcal{T}^{ {\bar i}}$ wihout changing it. The resulting space, however, is identical to
$\mathcal{T}^{\mathcal{R}}_{\left(\bar x,\bar y\right)}$, a contradiction.
Without loss of generality, we assume 
$\xi^{\bar i}_{n+\bar i}=1$.
Further, by straightforward application of the implicit function theorem and due to Lemma \ref{lem:QandE}a) and \ref{lem:QandE}b), we find a sequence of vectors $\xi^{\bar i}_t \in \mathcal{T}^{ \bar i}_{t}$ that converges to $\xi^{\bar i}$ for $t \to 0$. 
For this, we define
\[
\mathcal{T}^{\bar i}_t=\left\{
\xi \in \R^{2n}\,\left\vert\, \begin{array}{l} \begin{pmatrix}
Dh_p(x^t),0\end{pmatrix} \xi=0, p \in P, 
\begin{pmatrix}
Dg_q(x^t),0\end{pmatrix}\xi=0,q \in Q_0(x^t),\\
\begin{pmatrix}
0,e
\end{pmatrix}\xi=0,
\begin{pmatrix}
0,e_i
\end{pmatrix}\xi=0, i\in \mathcal{E}(y^t),\\
\begin{pmatrix}
e_i,0
\end{pmatrix}\xi=0, i \in a_{00}(\bar x,\bar y)\backslash\{\bar i\},
\begin{pmatrix}
0,e_i
\end{pmatrix}\xi=0, i \in a_{00}(\bar x,\bar y)\backslash\{\bar i\},\\
\begin{pmatrix}
e_{\bar i},\frac{x_{\bar i}^t}{y_{\bar i}^t} e_{\bar i}
\end{pmatrix}\xi=0,
\begin{pmatrix}
y^t_ie_i,x^t_ie_i
\end{pmatrix}\xi=0, i \in a_{01}(\bar x,\bar y)\cup a_{10}(\bar x,\bar y)
\end{array}
\right.\right\}.
\]
We again have, due to continuity arguments, that 
$\xi^{\bar i}_{t,n+\bar i} \not =0$.
For $\bar i \in a_{00}^y\left(\bar x,\bar y\right)$ we proceed
analogously by
considering $\mathcal{T}^{\mathcal{R}}_{\left(\bar x,\bar y\right)}$ again and replace two of the  involved equations 
$\begin{pmatrix}
e_{\bar i},0
\end{pmatrix}\xi=0$ and 
$\begin{pmatrix}
0,e_{\bar i}
\end{pmatrix}\xi=0$ by the equation
$\begin{pmatrix}
\lim\limits_{t \to 0}\frac{y_{\bar i}^t}{x_{\bar i}^t} e_{\bar i},e_{\bar i}
\end{pmatrix}\xi=0$. 
By the same arguments as before, we find
$\xi^{\bar i} \in \mathcal{T}^{\bar i}$ with
$\xi^{\bar i}_{\bar i} \ne 0$. Again we will assume $\xi^{\bar i}_{\bar i} = 1$ and find a sequence of vectors $\xi^{\bar i}_t \in \mathcal{T}^{ \bar i}_{t}$ that converges to $\xi^{\bar i}$ for $t \to 0$. Due to continuity, it holds then $\xi^{\bar i}_{t,\bar i} \ne 0$.

It is straightforward to verify the following observations for $t$ sufficiently small:

    a) Let $\left\{\xi^{\prime,1},\ldots,\xi^{\prime,\ell}\right\}$ be a base of $\mathcal{T}^{\prime}$, cf. Step 2b, then
    $\left\{\xi^{\prime,1},\ldots,\xi^{\prime,\ell}\right\}\cup \left\{\xi^{\bar i}_{t}\,\left\vert\,\bar i \in a_{00}(\bar x,\bar y)\right.\right\}$
is a set of linear independent vectors.
In fact, suppose for some coefficients $b_{i}, \in \R$, $i \in a_{00}\left(\bar x,\bar y\right),\beta_i \in \R, i=1,\ldots,\ell$ it holds:
\[
\sum\limits_{i \in a_{00}\left(\bar x,\bar y\right)} b_{i}\xi^{i}_{t}
+\sum\limits_{i=1}^{\ell}\beta_i\xi^{\prime,i}=0.
\]
For $\bar i \in a_{00}^x\left(\bar x,\bar y\right)$ we consider the $(n+\bar i)$-th row of this sum
\[
b_{\bar i}\underbrace{\xi^{\bar i}_{t,n+\bar i}}_{\ne 0}+\sum\limits_{i \in a_{00}\left(\bar x,\bar y\right)\backslash\left\{\bar i\right\}} b_{i}\underbrace{\xi^{i}_{t,n+\bar i}}_{=0}+\sum\limits_{i=1}^{\ell}\beta_i\underbrace{\xi^{\prime,i}_{n+\bar i}}_{=0}=0.
\]
If instead $\bar i \in a_{00}^y\left(\bar x,\bar y\right)$ we consider the $\bar i$-th row of the sum
\[
b_{\bar i}\underbrace{\xi^{\bar i}_{t,\bar i}}_{\ne 0}+\sum\limits_{i \in a_{00}\left(\bar x,\bar y\right)\backslash\left\{\bar i\right\}} b_{i}\underbrace{\xi^{i}_{t,\bar i}}_{=0}+\sum\limits_{i=1}^{\ell}\beta_i\underbrace{\xi^{\prime,i}_{\bar i}}_{=0}=0.
\]
Altogether, it must hold $b_{i} =0$, $i \in a_{00}\left(\bar x,\bar y\right)$. However, this implies
\[
\sum\limits_{i=1}^{\ell}\beta_i\xi^{\prime,i}=0.
\]
Hence, $\beta_i=0$ for $i=1,\ldots,\ell$.

    b) It holds $\xi^{\bar i}_{t} \in \mathcal{T}^\mathcal{S}_{\left(x^t,y^t\right)}$, cf. Step 1a, for any $\bar i \in a_{00}(\bar x,\bar y)$. 
    
    c) It holds $\left(\eta_i^{\le,t}-\eta_i^{\ge,t}\right)\xi^{\bar i}_{t,i}\xi^{\bar i}_{t,n+i} \leq 0, i \in \mathcal{H}\left(x^t,y^t\right)$. Since $\xi^{\bar i}_{t} \in \mathcal{T}^\mathcal{S}_{\left(x^t,y^t\right)}$, we obtain:
    \[
\begin{array}{rcl}
\left(\eta_i^{\le,t}-\eta_i^{\ge,t}\right)\xi^{\bar i}_{t,i}\xi^{\bar i}_{t,n+i}
&=&\left(\eta_i^{\ge,t}-\eta_i^{\le,t}\right)\left(\xi^{\bar i}_{t,n+i}\right)^2\frac{x^t_i}{y^t_i}.
\end{array}
\]
If $i\in \mathcal{H}^{\ge}\left(x^t,y^t\right)$,
we have $x_i^t<0,y_i^t>0$ and $\eta_i^{\le,t}=0$. Moreover, due to ND2, we have
$\eta_i^{\ge,t}>0$. The assertion follows immediately. The other case $i\in \mathcal{H}^{\le}\left(x^t,y^t\right)$ is completely analogous.

    d) It holds $\lim\limits_{t\to 0}\left(\eta_{\bar i}^{\le,t}-\eta_{\bar i}^{\ge,t}\right)\xi^{\bar i}_{t,\bar i}\xi^{\bar i}_{t,n+\bar i}=-\infty$. 
    We calculate:
\[
\begin{array}{rcl}
\left(\eta_{\bar i}^{\le,t}-\eta_{\bar i}^{\ge,t}\right)\xi^{\bar i}_{t,{\bar i}}\xi^{\bar i}_{t,n+{\bar i}}
&=&\left\{
\begin{array}{ll}
   \left(\eta_{\bar i}^{\ge,t}-\eta_{\bar i}^{\le,t}\right)\left(\xi^{\bar i}_{t,n+\bar i}\right)^2\frac{x^t_{\bar i}}{y^t_{\bar i}}& \mbox{for } \bar i \in a_{00}^x(\bar x,\bar y),\\
   \left(\eta_{\bar i}^{\ge,t}-\eta_{\bar i}^{\le,t}\right)\left(\xi^{\bar i}_{t,\bar i}\right)^2\frac{y^t_{\bar i}}{x^t_{\bar i}}& \mbox{for } \bar i \in a_{00}^y(\bar x,\bar y).
\end{array}\right.
\end{array}
\]
Let us suppose $\bar i \in a_{00}^x(\bar x,\bar y)$. We have $y^t_{\bar i}\ne 0$ and, thus, $\nu^t_{\bar i}=0$. We use Remark \ref{rem:multipliers} and NDT3 to conclude that the sequence
$ \left(\eta_{\bar i}^{\ge,t}-\eta_{\bar i}^{\le,t}\right)x^t_{\bar i}$ converges to $\varrho_{2,\bar i}<0$ for $t \to 0$. Further,
$\frac{1}{y^t_{\bar i}}>0$ tends to infinity for $t \to 0$. Finally,
$\left(\xi^{\bar i}_{t,n+\bar i}\right)^2$ converges to $1$ for $t \to 0$, due to the construction of
$\xi^{\bar i}_t$.
Thus, the assertion follows.
Instead, let us suppose $\bar i \in a_{00}^y(\bar x,\bar y)$. 
This time, we have that
$\left(\xi^{\bar i}_{t,\bar i}\right)^2$ converges to $1$ for $t \to 0$.
Due to Remark \ref{rem:multipliers} and NDT3, $\left(\eta_{\bar i}^{\ge,t}-\eta_{\bar i}^{\le,t}\right)y^t_{\bar i}$ converges to $\varrho_{1,\bar i}\ne 0$ for $t \to 0$.
If $\bar i\in \mathcal{H}^{\ge}\left(x^t,y^t\right)$, then $\varrho_{1,\bar i}$ is positive from here. Also, $\frac{1}{x^t_{\bar i}}<0$ tends to minus infinity for $t \to 0$. The other case $\bar i\in \mathcal{H}^{\le}\left(x^t,y^t\right)$ is completely analogous. 

e) We notice that $\displaystyle \xi^{\bar i ^T}_{t} D^2 L^\mathcal{R}\left( x^t, y^t\right)\xi^{\bar i}_{t}$ converges for $t \to 0$ due to the construction above.

Finally, for $\bar i\in  a_{00}(\bar x,\bar y)$ we estimate:
\[
\begin{array}{rcl}
\xi_t^{\bar i ^T} D^2 L^\mathcal{S}\left( x^t, y^t\right)\xi_t^{\bar i}&\overset{(\ref{eq:Hessetandbar})}{=}&
\displaystyle \xi^{\bar i ^T}_{t} D^2 L^\mathcal{R}\left(x^t, y^t\right)\xi^{\bar i}_{t}\\ && \displaystyle
- \sum\limits_{i\in \mathcal{H}^\ge\left(x^t,y^t\right)} 2 \eta_i^{\ge,t} \xi_{t,i}^{\bar i}  \xi_{t,n+i}^{\bar i}+
\sum\limits_{i\in \mathcal{H}^\le\left(x^t,y^t\right)} 2 \eta_i^{\le,t}\xi_{t,i}^{\bar i}\xi_{t,n+i}^{\bar i} \\
&= &
\displaystyle \xi^{\bar i ^T}_{t} D^2 L^\mathcal{R}\left(x^t, y^t\right)\xi^{\bar i}_{t}
 +2\sum\limits_{i\in \mathcal{H}\left(x^t,y^t\right)} \left( \eta_i^{\le,t}-\eta_i^{\ge,t} \right)\xi_{t,i}^{\bar i}  \xi_{t,n+i}^{\bar i}\\ 
&\overset{c)}{\le}&\displaystyle \xi^{\bar i ^T}_{t} D^2 L^\mathcal{R}\left(x^t, y^t\right)\xi^{\bar i}_{t}+2\left(\eta_{\bar i}^{\le,t}-\eta_{\bar i}^{\ge,t}\right)\xi^{\bar i}_{t,\bar i}\xi^{\bar i}_{t,n+\bar i}.
\end{array}
\]
Thus, due to d) and e), $\xi^{\bar i ^T}_{t} D^2 L^\mathcal{S}\left(x^t, y^t\right) \xi^{\bar i}_{t}$ has to be negative for $t$ small enough. 
As we have seen in Step 2c, it holds $\mathcal{T}^\prime \subset \mathcal{T}^\mathcal{S}_{\left(x^t,y^t\right)}$. Then, due a) and b), we have therefore:
\[
QI^\mathcal{S}_{t,\mathcal{T}^\mathcal{S}_{\left(x^t,y^t\right)}}
\geq QI^\mathcal{S}_{t,\mathcal{T}^\prime} + \left\vert a_{00}(\bar x,\bar y)\right\vert.
\]
By Step 2a, we have $QI^\mathcal{S}_{t,\mathcal{T}^\mathcal{R}_{\left(\bar x,\bar y\right)}}=\overline{QI}^\mathcal{R}_{\mathcal{T}^\mathcal{R}_{\left(\bar x,\bar y\right)}}$, and by Step 2b, $QI^\mathcal{S}_{t,\mathcal{T}^\mathcal{R}_{\left(\bar x,\bar y\right)}}=QI^\mathcal{S}_{t,\mathcal{T}^\prime}$. Overall, we obtain:
\[
m=QI^\mathcal{S}_{t,\mathcal{T}^\mathcal{S}_{\left(x^t,y^t\right)}}
\ge
\overline{QI}^\mathcal{R}_{\mathcal{T}^\mathcal{R}_{\left(\bar x,\bar y\right)}}+\left\vert a_{00}(\bar x,\bar y)\right\vert=TI.
\]
\end{proof}

Let us illustrate the necessity of NDT6 for the validity of Theorem \ref{thm:kktsequence}.

\begin{example}[Necessity of NDT6]
We consider the following Scholtes-type regularization $\mathcal{S}$ with $n=2$, $s=1$ and $0 < c_1 < c_2$:
\[
\begin{array}{rl}
\mathcal{S}: \quad \min\limits_{x,y}& (1+x_1)^2+(3-2x_2)^2+c_1y_1+c_2y_2\\
\mbox{s.t.} &x_1+x_2-1\ge 0, \\ &y_1+y_2\ge 1, \quad  
-t\le x_i y_i \le t,  \quad 0\leq y_i\le 1+\varepsilon, \quad i=1,2,
\end{array}
\]
as well as the point $\left(x^t,y^t\right)=(0,1,1,0)$.
We claim that this point is a nondegenerate Karush-Kuhn-Tucker point. Indeed, it holds:
\[
\begin{pmatrix}
2\\2\\c_1\\c_2
\end{pmatrix}=
\mu_{1,1}^t
\begin{pmatrix}
1\\1\\0\\0
\end{pmatrix}
+\mu_3^t
\begin{pmatrix}
0\\0\\1\\1
\end{pmatrix}
+\nu_2^t
\begin{pmatrix}
0\\0\\0\\1
\end{pmatrix}
\]
with the positive multipliers $\mu_{1,1}^t=2,\mu_3^t=c_1,\nu_2^t=c_2-c_1$. Obviously, LICQ and strict complementarity, i.e. ND1 and ND2, respectively, are fulfilled. We show that $D^2 L^\mathcal{S}(x^t,y^t)\restriction_{\mathcal{T}^{\mathcal{S}}_{(x^t,y^t)}}$ is nonsingular and calculate the number of its negative eigenvalues. The tangent space is
$\mathcal{T}^{\mathcal{S}}_{(x^t,y^t)}=\left\{\xi \in \R^{4}\,\left\vert\, 
\xi_1+\xi_2=0, \xi_3=\xi_4=0
\right.\right\}$.
For the Hessian of the corresponding Lagrange function we have:

\[
D^2 L^\mathcal{S}(x^t,y^t)
=\begin{pmatrix}2&0&0&0\\
0&-4&0&0\\
0&0&0&0\\
0&0&0&0
\end{pmatrix}.
\]
Thus, for $\xi \in  \mathcal{T}^{\mathcal{S}}_{(x^t,y^t)}$ it holds:
\[
\xi^T D^2 L^\mathcal{S}(x^t,y^t) \xi=2\xi_1^2-4\xi_2^2=-2\xi_1^2.
\]
Hence, ND3 is also fulfilled, the Karush-Kuhn-Tucker point $\left(x^t,y^t\right)$ is nondegenerate and its quadratic index equals one, i.e. $m=1$ in Theorem \ref{thm:kktsequence}. The limiting point is $(\bar x,\bar y)=(0,1,1,0)$. This point is T-stationary for the corresponding regularized continuous reformulation $\mathcal{R}$ according to Theorem \ref{thm:regul}, since MPOC-LICQ is fulfilled.
Indeed, we have:
\[
\begin{pmatrix}
2\\2\\c_1\\c_2
\end{pmatrix}=
\bar \mu_{1,1}
\begin{pmatrix}
1\\1\\0\\0
\end{pmatrix}
+\bar \mu_3
\begin{pmatrix}
0\\0\\1\\1
\end{pmatrix}
+\bar \sigma_{1,1}
\begin{pmatrix}
1\\0\\0\\0
\end{pmatrix}
+\bar \sigma_{2,2}
\begin{pmatrix}
0\\0\\0\\1
\end{pmatrix}
\]
with the unique multipliers $\bar \mu_{1,1}=2,\bar\mu_3=c_1,\bar \sigma_{1,1}=0,\bar \sigma_{2,2}=c_2-c_1.$
It is easy to see that this point is nondegenerate with vanishing T-index, i.e. $TI=0$, since $a_{00}(\bar x,\bar y)=\emptyset$ and $\mathcal{T}^{\mathcal{R}}_{(\bar x,\bar y)}=\{0\}$. Note that additionally $\left\{i\in a_{01}\left(\bar x,\bar y\right)\,\left\vert\,\bar \sigma_{1,i}=0\right.\right\}=\{1\}$.
Although all assumptions of Theorem \ref{thm:kktsequence} are fulfilled, we have here:
\[
   TI=\max\left\{m - \left\vert\left\{i\in a_{01}\left(\bar x,\bar y\right)\,\left\vert\,\bar \sigma_{1,i}=0\right.\right\} \right\vert, 0\right\}.
\]
With other words, the saddle points of the Scholtes-type regularization $\mathcal{S}$ approximate a minimizer of the regularized continuous reformulation $\mathcal{R}$. The reason is that the $\sigma$-multipliers corresponding to zero $x$- and nonzero $y$-variables vanish.
The lower bound given in Theorem \ref{thm:kktsequence} is attained.
\qed
\end{example}

Next, we point out that the assumption NDT6 is not restrictive at all.

\begin{remark}[Genericity for NDT6]
\label{rem:ndt6}
    Let us briefly sketch why condition NDT6 must be generically fulfilled at the T-stationary points of $\mathcal{R}$. First, we note that all T-stationary points of $\mathcal{R}$ are generically nondegenerate, see \cite{laemmel:reform}. 
    Now, let us count the losses of freedom induced by the definition of a T-stationary point. For feasibility we have $\left\vert P\right\vert$ equality constraints, $\left\vert Q_0\right\vert$ active inequality constraints, $\left\vert\mathcal{E}\right\vert$ bounding constraints on the $y$-variables, eventually one summation constraint, and $\left\vert a_{01}\right\vert+\left\vert a_{10}\right\vert+2\left\vert a_{00}\right\vert$ orthogonality type constraints. Additional losses of freedom come from the T-stationarity condition. They amount to $2n-\left\vert P\right\vert-\left\vert Q_0\right\vert-\left\vert\mathcal{E}\right\vert-1-\left\vert a_{01}\right\vert-\left\vert a_{10}\right\vert-2\left\vert a_{00}\right\vert$ if the summation constraint is active, and to $2n-\left\vert P\right\vert-\left\vert Q_0\right\vert-\left\vert\mathcal{E}\right\vert-\left\vert a_{01}\right\vert-\left\vert a_{10}\right\vert-2\left\vert a_{00}\right\vert$ otherwise. 
In both cases, the losses of freedom are equal to the number of variables $2n$. The violation of NDT6 would produce an additional loss of freedom, which would imply that the total available degrees of freedom $2n$ are exceeded. By virtue
of the structured jet transversality theorem from \cite{guenzel:2008}, this cannot happen generically. \qed
\end{remark}

Now, we prove that the Scholtes-type regularization method is well-defined.

\begin{theorem}[Well-posedness of $\mathcal{S}$ from $\mathcal{R}$]
   Let $(\bar x, \bar y)$ be a nondegenerate T-stationary point of $\mathcal{R}$ with T-index $m$, additionally, fulfilling NDT6. Then, for all sufficiently small $t$ there exists a nondegenerate Karush-Kuhn-Tucker point $(x^t,y^t)$ of $S$ within a neighborhood of $(\bar x, \bar y)$, which has the same quadratic index $m$. 
\end{theorem}

\begin{proof}
First, we show that for all $i\in a_{10}\left(\bar x,\bar y\right)$ it holds $\bar \sigma_{2,i}\ne 0$. 
Assume contrarily that $\bar \sigma_{2,\bar i}= 0$ for some $\bar i \in a_{10}\left(\bar x,\bar y\right)$. We then have due to T-stationarity, cf. (\ref{eq:k-throw}):
\[
c_{\bar i}=\bar \sigma_{2,\bar i}+\bar \mu_3=\bar \mu_3.
\]
Moreover, we have in view of Lemma \ref{lem:a01}c) an index $\tilde i \in a_{01}\left(\bar x, \bar y\right)\backslash \mathcal{E}\left(\bar y\right).$ Thus it holds, cf. (\ref{eq:k-throw}),
$c_{\tilde i}=\bar \mu_3$.
Due to the assumption on $c$, we have $\bar i=\tilde i$, a contradiction. Hence, we may write:
\[
a_{10}\left(\bar x,\bar y\right)=
\left\{i\in a_{10}\left(\bar x,\bar y\right)\,\left\vert\,\bar \sigma_{2,i}<0\right.\right\}\cup\left\{i\in a_{10}\left(\bar x,\bar y\right)\,\left\vert\,\bar \sigma_{2,i}>0\right.\right\}= a_{10}^<\left(\bar x,\bar y\right) \cup a_{10}^>\left(\bar x,\bar y\right).
\]
Due to NDT6 and NDT3, we may split the other index sets as follows:
\[
a_{01}\left(\bar x,\bar y\right)=
\left\{i\in a_{01}\left(\bar x,\bar y\right)\,\left\vert\,\bar \sigma_{1,i}<0\right.\right\}\cup\left\{i\in a_{01}\left(\bar x,\bar y\right)\,\left\vert\,\bar \sigma_{1,i}>0\right.\right\} = a_{01}^<\left(\bar x,\bar y\right) \cup a_{01}^>\left(\bar x,\bar y\right),
\]
\[
a_{00}\left(\bar x,\bar y\right)=
\left\{i\in a_{00}\left(\bar x,\bar y\right)\,\left\vert\,\bar \varrho_{1,i}<0\right.\right\}\cup\left\{i\in a\left(\bar x,\bar y\right)\,\left\vert\,\bar \varrho_{1,i}>0\right.\right\}= a_{00}^<\left(\bar x,\bar y\right) \cup a_{00}^>\left(\bar x,\bar y\right).
\]

We consider the auxiliary system of equations $F(t,x,y,\lambda,\mu,\sigma,\varrho) =0$ given by (\ref{eq:ift-stat})-(\ref{eq:ift-f5}), which mimics stationarity and feasibility. For stationarity we use:
\begin{equation}
\label{eq:ift-stat}
- L\left(t,x,y,\lambda,\mu,\sigma,\varrho\right)=0,
\end{equation}
where
\[
   \begin{array}{rcl}
    L&=&
  \begin{pmatrix}
   \nabla f(x)\\
   c
  \end{pmatrix}- \displaystyle\sum\limits_{p \in P} \lambda_p 
  \begin{pmatrix}
  \nabla h_p(x)\\
  0
  \end{pmatrix}-
    \sum\limits_{q \in Q_0(\bar x)}\mu_{1,q} 
    \begin{pmatrix}
    \nabla g_q(x)\\
    0
    \end{pmatrix}+
    \sum\limits_{i\in\mathcal{E}(\bar y)} \mu_{2,i} \begin{pmatrix}
    0\\
    e_i
\end{pmatrix}\\ 
    && \displaystyle -
 \mu_3 \begin{pmatrix}
0\\
e
\end{pmatrix} -\sum\limits_{i \in a_{01}\left(\bar x,\bar y\right)} \frac{\sigma_{1,i}}{\bar y_{i}}
    \begin{pmatrix}
    y_{i}e_{ i}\\
    x_{i}e_{ i}
    \end{pmatrix}
    -\sum\limits_{i \in a_{10}^<\left(\bar x,\bar y \right)}   \frac{\sigma_{2,i}}{\bar x_{i}}  \begin{pmatrix}
    y_{i}e_{ i}\\
    x_{i}e_{ i}
    \end{pmatrix}
     \\ 
    &&\displaystyle -\sum\limits_{i \in a_{10}^>\left(\bar x,\bar y \right) }  \sigma_{2,i}    \begin{pmatrix}
    0\\
    e_{ i}
    \end{pmatrix}-\sum\limits_{i \in a_{00}\left(\bar x, \bar y\right)} \left( \varrho_{1,i}
    \begin{pmatrix}
    e_{ i}\\
   0
    \end{pmatrix}
   + \varrho_{2,i}
    \begin{pmatrix}
    0\\
    e_{ i}
    \end{pmatrix}\right).
    \end{array}
\]
For feasibility we use:
\begin{equation}
\label{eq:ift-f1}
      h_p(x)=0, p\in P, \quad
     g_q(x)=0, q \in Q_0\left(\bar x\right),
\end{equation}
\begin{equation}
\label{eq:ift-f1a}
    1+\varepsilon-y_i=0, i \in \mathcal{E}\left(\bar y\right), \quad  
   \displaystyle \sum_{i=1}^n y_i -(n-s)=0,
\end{equation}
\begin{equation}
\label{eq:ift-f2}
\frac{1}{\bar y_i}\left(x_iy_i-t\right)=0, i\in a_{01}^<\left(\bar x,\bar y\right), \quad
    \frac{1}{\bar y_i}\left(x_iy_i+t\right)=0, i\in a_{01}^>\left(\bar x,\bar y\right),     
\end{equation}
\begin{equation}
\label{eq:ift-f3}
\frac{1}{\vert\bar x_i \vert}\left(\mbox{sgn}\left(\bar x_i\right)x_iy_i-t\right)=0, i \in  a_{10}^<\left(\bar x, \bar y\right), \quad y_i=0, i \in a_{10}^>\left(\bar x, \bar y\right),   
\end{equation}
\begin{equation}
\label{eq:ift-f4}
   x_i+\displaystyle \frac{\varrho_{2,i}\sqrt{t}}{\sqrt{\varrho_{1,i}\varrho_{2,i}}}=0,
   i \in a_{00}^<\left(\bar x,\bar y\right), \quad
   x_i-\displaystyle \frac{\varrho_{2,i}\sqrt{t}}{\sqrt{-\varrho_{1,i} \varrho_{2,i}}}=0, 
    i \in a_{00}^>\left(\bar x,\bar y\right),
\end{equation}
\begin{equation}
\label{eq:ift-f5}
     y_i+\displaystyle \frac{\varrho_{1,i}\sqrt{t}}{\sqrt{\varrho_{1,i} \varrho_{2,i}}}=0,
      i \in a_{00}^<\left(\bar x,\bar y\right), \quad 
   y_i-\displaystyle \frac{\varrho_{1,i}\sqrt{t}}{\sqrt{-\varrho_{1,i} \varrho_{2,i}}}=0, i \in  a_{00}^>\left(\bar x,\bar y\right).
\end{equation}
In view of feasibility and T-stationarity of $\left(\bar x,\bar y\right)$ for $\mathcal{R}$, the vector
$(0,\bar x,\bar y,\bar \lambda,\bar \mu,\bar \sigma,\bar \varrho)$ solves (\ref{eq:ift-stat})-(\ref{eq:ift-f5}).
We consider the Jacobian matrix
$
D F (t,x,y,\lambda,\mu, \sigma, \rho)=\begin{bmatrix}
A&B\\
B^T&D
\end{bmatrix}$,
where 

\[
\begin{array}{rcl}
      A&=& \displaystyle
    -D^2f(x)
+\sum\limits_{p\in P}\lambda_p
    D^2h_p(x)
+\sum\limits_{q\in Q_0\left(\bar x\right)}\mu_{1,q}
D^2g_q(x) \\
&& \displaystyle +
\sum\limits_{i\in a_{01}\left(\bar x,\bar y\right)}\frac{1}{\bar y_{i}}E(i) +
\sum\limits_{i\in a_{10}^<\left(\bar x,\bar y\right)}\frac{1}{\bar x_{i}}E(i),
\end{array}
\]
the columns of $B$ are give by the vectors:
\[
\begin{pmatrix}
    \nabla h_p(x)\\0
\end{pmatrix},p \in P,\quad
\begin{pmatrix}
    \nabla g_q(x)\\0
\end{pmatrix},q\in Q_0\left(\bar x\right),\quad
\begin{pmatrix}
    0\\-e_i
\end{pmatrix},i \in \mathcal{E}\left(\bar y\right),\quad
\begin{pmatrix}
    0\\e
\end{pmatrix},
\]
\[
\begin{pmatrix}
   \frac{y_{i}}{\bar y_i}e_{i}\\
   \frac{x_{i}}{\bar y_i}e_{i}
\end{pmatrix},i \in a_{01}\left(\bar x,\bar y\right),\quad
\begin{pmatrix}
  \frac{y_{i}}{\bar x_i}e_{i}\\
  \frac{x_{i}}{\bar x_i}e_{i}
\end{pmatrix},i \in a_{10}^<\left(\bar x,\bar y\right), \quad
\begin{pmatrix}
0\\e_{i}
\end{pmatrix},i \in a_{10}^>\left(\bar x,\bar y\right),
\]
\[
\begin{pmatrix}
   e_{i}\\
   0
\end{pmatrix}, 
\begin{pmatrix}
  0\\
  e_{i}
\end{pmatrix},i \in a_{00}\left(\bar x,\bar y\right),
\]
and $D$ consists of $\vert P \vert+\left\vert Q_0\left(\bar x\right)\right\vert+\left\vert\mathcal{E}\left(\bar y\right)\right\vert+1+\left\vert a_{01}\left(\bar x,\bar y\right)\right\vert+\left\vert a_{10}\left(\bar x,\bar y\right)\right\vert$ vanishing rows. The remaining rows of $D$ are given by the vectors:
\[
\begin{pmatrix}
0\\
    -\frac{\sqrt{-\varrho_{2,i}t}}{2\sqrt{-\varrho_{1,i}^3}}e_i\\
    \frac{\sqrt{t}}{2\sqrt{\varrho_{1,i}\varrho_{2,i}}}e_i
\end{pmatrix}, i \in a_{10}^<\left(\bar x,\bar y\right),\quad
\begin{pmatrix}
0\\
    \frac{\sqrt{-\varrho_{2,i}t}}{2\sqrt{\varrho_{1,i}^3}}e_i\\
    -\frac{\sqrt{t}}{2\sqrt{-\varrho_{1,i}\varrho_{2,i}}}e_i
\end{pmatrix}, i \in a_{10}^>\left(\bar x,\bar y\right),
\]
\[
\begin{pmatrix}
0\\
    \frac{\sqrt{t}}{2\sqrt{\varrho_{1,i}\varrho_{2,i}}}e_i\\
    -\frac{\sqrt{-\varrho_{1,i}t}}{2\sqrt{-\varrho_{2,i}^3}}e_i
\end{pmatrix}, i \in a_{10}^<\left(\bar x,\bar y\right),\quad
\begin{pmatrix}
0\\
-\frac{\sqrt{t}}{2\sqrt{-\varrho_{1,i}\varrho_{2,i}}}e_i\\
    -\frac{\sqrt{\varrho_{1,i}t}}{2\sqrt{-\varrho_{2,i}^3}}e_i
\end{pmatrix}, i \in a_{10}^>\left(\bar x,\bar y\right).
\]
Additionally we have $D=0$ at $(0,\bar x,\bar y,\bar \lambda,\bar \mu,\bar \sigma,\bar \varrho)$. Hence, we can apply Theorem 2.3.2 from \cite{Jongen:2004}, which says that
$
D F (0,\bar x, \bar y, \bar \lambda, \bar \mu, \bar \sigma, \bar \rho)=
\begin{bmatrix}
A&B\\
B^T&0
\end{bmatrix}$ is nonsingular if and only if
$\xi^TA\xi\ne0$ for all $\xi \in B^\perp$, the orthogonal complement of the subspace spanned by the columns of $B$. In view of $B^\perp=\mathcal{T}^{\mathcal{R}}_{\left(\bar x,\bar y\right)}$, we check:
\[
\begin{array}{rcl}
\xi^TA\xi&=&\xi^T\left(-D^2L^\mathcal{R}\left(\bar x,\bar y\right)+\displaystyle \sum\limits_{i\in a_{01}\left(\bar x,\bar y\right)}\frac{1}{\bar y_{i}}E(i) +
\sum\limits_{i\in a_{10}^<\left(\bar x,\bar y\right)}\frac{1}{\bar x_{i}}E(i)\right)\xi\\
&=&-\underbrace{\xi^TD^2L^\mathcal{R}\left(\bar x,\bar y\right)\xi}_{\ne 0\mbox{\footnotesize \ due to NDT4}}
+\displaystyle \sum\limits_{i\in a_{01}\left(\bar x,\bar y\right)}\frac{1}{\bar y_{i}}2\xi_{i}\underbrace{\xi_{n+i}}_{=0}+
\sum\limits_{i\in a_{10}^<\left(\bar x,\bar y\right)}\frac{1}{\bar x_{i}}
2\xi_{i}\underbrace{\xi_{n+i}}_{=0}.
\end{array}
\]
Hence, by means of the implicit function theorem we obtain for any $t>0$ sufficiently small a solution $\left(t,x^t,y^t,\lambda^t,\mu^t,\sigma^t,\varrho^t\right)$ of the system of equations (\ref{eq:ift-stat})-(\ref{eq:ift-f5}).

By choosing $t$ even smaller, if necessary, we can ensure due to continuity reasons as well as  NDT2, NDT3, and NDT6  that the following holds:

(i) $g_q\left(x^t\right)>0$, $q \in Q\backslash Q_0\left(\bar x\right)$ and   $\mu_{1,q}^t>0$, $q \in Q_0\left(\bar x\right)$,

   (ii) $ \mbox{sgn}\left(\sigma^t_{1,i}\right)=\mbox{sgn}\left(\bar \sigma_{1,i}\right)$, $i\in a_{01}\left(\bar x,\bar y\right)$,
   
    (iii)
  $ \mbox{sgn}\left(\sigma^t_{2,i}\right)=\mbox{sgn}\left(\bar \sigma_{2,i}\right)$, $\mbox{sgn}\left(x^t_{i}\right)= \mbox{sgn}\left(\bar x_{i}\right)$, $i \in a_{10}\left(\bar x,\bar y\right)$,
  
    (iv) $ \mbox{sgn}\left(\varrho^t_{1,i}\right)=\mbox{sgn}\left(\bar \varrho_{1,i}\right)$,
  $\varrho_{2,i}^t<0$, $i\in a_{00}\left(\bar x,\bar y\right)$,
    
    (v) $y_i^t\ge 0$, $i \in \left\{1,\ldots,n\right\}$ and $y^t_i<1+\varepsilon$, $i\in \left\{1,\ldots,n\right\}\backslash\mathcal{E}\left(\bar y\right)$.

\noindent
From here it is straightforward to see that $\left(x^t,y^t\right)$ is feasible for $\mathcal{S}$ and we have:

    (i) $Q_0\left(x^t\right)=Q_0\left(\bar x\right)$,
  $\mathcal{E}\left(y^t\right)=\mathcal{E}\left(\bar y\right)$,
    $\mathcal{N}\left(y^t\right)= a^>_{10}\left(\bar x, \bar y\right)$, $\mathcal{O}\left(x^t,y^t\right)=\emptyset$,
    
    (ii) $\mathcal{H}^{\ge}\left(x^t,y^t\right)=a^>_{01}\left(\bar x,\bar y\right)\cup \left\{i \in a^<_{10}\left(\bar x, \bar y\right)\,\left\vert\, x^t_{i} <0\right.\right\} \cup 
    a^>_{00}\left(\bar x,\bar y\right)$,
    
    (iii) $\mathcal{H}^{\le}\left(x^t,y^t\right)=a^<_{01}\left(\bar x,\bar y\right)\cup \left\{i \in a^<_{10}\left(\bar x, \bar y\right)\,\left\vert\, x^t_{i}>0\right.\right\} \cup a^<_{00}\left(\bar x,\bar y\right)$.

Thus, it holds:
\[
   \begin{array}{rcl}
  \begin{pmatrix}
   \nabla f\left(x^t\right)\\
   c
  \end{pmatrix}&=& \displaystyle\sum\limits_{p \in P} \lambda_{p}^t 
  \begin{pmatrix}
  \nabla h_p\left(x^t\right)\\
  0
  \end{pmatrix}+
    \sum\limits_{q \in Q_0(\bar x)}\mu_{1,q}^t 
    \begin{pmatrix}
    \nabla g_q\left(x^t\right)\\
    0
    \end{pmatrix}-
    \sum\limits_{i\in\mathcal{E}(\bar y)} \mu_{2,i}^t \begin{pmatrix}
    0\\
    e_i
\end{pmatrix}\\ 
    && +
 \mu_3^t \begin{pmatrix}
0\\
e
\end{pmatrix}\displaystyle +\sum\limits_{i \in a_{01}\left(\bar x,\bar y\right)} \sigma_{1,i}^t\frac{1}{\bar y_{i}}
    \begin{pmatrix}
    y_{i}^te_{ i}\\
    x_{i}^te_{ i}
    \end{pmatrix}
    +\sum\limits_{i \in a_{10}^<\left(\bar x,\bar y \right)}  \sigma_{2,i}^t  \frac{1}{\bar x_{i}}  \begin{pmatrix}
    y_{i}^te_{ i}\\
    x_{i}^te_{ i}
    \end{pmatrix}\\ 
    &&\displaystyle
     +\sum\limits_{i \in a^>_{10}\left(\bar x,\bar y \right)}  \sigma_{2,i}^t    \begin{pmatrix}
    0\\
    e_{ i}
    \end{pmatrix}\displaystyle +\sum\limits_{i \in a_{00}\left(\bar x, \bar y\right)} \left( \varrho^t_{1,i}
    \begin{pmatrix}
     e_{ i}\\
   0
    \end{pmatrix}
   + \varrho^t_{2,i}
    \begin{pmatrix}
     0\\
    e_{ i}
    \end{pmatrix}\right)\\
&=&\displaystyle\sum\limits_{p \in P} \lambda_{p}^t 
  \begin{pmatrix}
  \nabla h_p\left(x^t\right)\\
  0
  \end{pmatrix}+
    \sum\limits_{q \in Q_0(x^t)}\mu_{1,q}^t 
    \begin{pmatrix}
    \nabla g_q\left(x^t\right)\\
    0
    \end{pmatrix}-
    \sum\limits_{i\in\mathcal{E}(y^t)} \mu_{2,i}^t \begin{pmatrix}
    0\\
    e_i
\end{pmatrix}
 \\ 
    && \displaystyle+
 \mu_3^t \begin{pmatrix}
0\\
e
\end{pmatrix}
+ \sum\limits_{\substack{i\in\mathcal{H}^\ge\left(x^t,y^t\right),\\
i \in a_{01}\left(\bar x, \bar y\right)}}\sigma_{1,i}^t\frac{1}{\bar y_{i}}\begin{pmatrix}
    y_i^te_i\\
   x_i^te_i
\end{pmatrix}
+ \sum\limits_{\substack{i\in\mathcal{H}^\ge\left(x^t,y^t\right),\\
i \in a_{10}\left(\bar x, \bar y\right)}}\sigma_{2,i}^t  \frac{1}{\bar x_{i}}  \begin{pmatrix}
    y_i^te_i\\
   x_i^te_i
\end{pmatrix} \\ &&\displaystyle
+ \sum\limits_{\substack{i\in\mathcal{H}^\ge\left(x^t,y^t\right),\\
i \in a_{00}\left(\bar x, \bar y\right)}}\displaystyle \frac{\sqrt{-\varrho^t_{1,i} \varrho^t_{2,i}}}{\sqrt{t}} \begin{pmatrix}
    y_i^te_i\\
   x_i^te_i
\end{pmatrix}+ \sum\limits_{\substack{i\in\mathcal{H}^\le\left(x^t,y^t\right),\\
i \in a_{01}\left(\bar x, \bar y\right)}}\sigma_{1,i}^t\frac{1}{\bar y_{i}}\begin{pmatrix}
    y_i^te_i\\
   x_i^te_i
\end{pmatrix}
\\ 
    && \displaystyle

+ \sum\limits_{\substack{i\in\mathcal{H}^\le\left(x^t,y^t\right),\\
i \in a_{10}\left(\bar x, \bar y\right)}}\sigma_{2,i}^t  \frac{1}{\bar x_{i}}  \begin{pmatrix}
    y_i^te_i\\
   x_i^te_i
\end{pmatrix}
+ \sum\limits_{\substack{i\in\mathcal{H}^\le\left(x^t,y^t\right),\\
i \in a_{00}\left(\bar x, \bar y\right)}}\displaystyle -\frac{\sqrt{\varrho^t_{1,i} \varrho^t_{2,i}}}{\sqrt{t}} \begin{pmatrix}
    y_i^te_i\\
   x_i^te_i
\end{pmatrix}
\\ 
&&+
    \displaystyle \sum\limits_{i\in\mathcal{N}\left( y^t\right)}  \sigma_{2,i}^t \begin{pmatrix}
    0\\
    e_i
\end{pmatrix}.
\end{array}
\]
We rename the multipliers as follows:
\[
\eta^{\ge,t}_i=\left\{
\begin{array}{ll}
\displaystyle \sigma_{1,i}^t\frac{1}{\bar y_{i}},& \mbox{for }i\in \mathcal{H}^\ge\left(x^t,y^t\right)\cap a_{01}\left(\bar x, \bar y\right),\\ 
\displaystyle \sigma_{2,i}^t\frac{1}{\bar x_{i}},& \mbox{for }i\in \mathcal{H}^\ge\left(x^t,y^t\right)\cap a_{10}\left(\bar x, \bar y\right),\\ 
\displaystyle \frac{\sqrt{-\varrho^t_{1,i} \varrho^t_{2,i}}}{\sqrt{t}},& \mbox{for }i\in \mathcal{H}^\ge\left(x^t,y^t\right)\cap a_{00}\left(\bar x, \bar y\right),\\
0,&\mbox{else,}
\end{array}\right.
\]
\[
\eta^{\le,t}_i=\left\{
\begin{array}{ll}
\displaystyle -\sigma_{1,i}^t\frac{1}{\bar y_{i}},& \mbox{for }i\in \mathcal{H}^\le\left(x^t,y^t\right)\cap a_{01}\left(\bar x, \bar y\right),\\ 
\displaystyle -\sigma_{2,i}^t\frac{1}{\bar x_{i}},& \mbox{for }i\in \mathcal{H}^\le\left(x^t,y^t\right)\cap a_{10}\left(\bar x, \bar y\right),\\ 
\displaystyle \frac{\sqrt{\varrho^t_{1,i}\varrho^t_{2,i}}}{\sqrt{t}},& \mbox{for }i\in \mathcal{H}^\le\left(x^t,y^t\right)\cap a_{00}\left(\bar x, \bar y\right),\\ 
0,&\mbox{else,}
\end{array}\right.
\]
\[
\nu^t_i=\left\{
\begin{array}{ll}
\displaystyle \sigma_{2,i}^t,& \mbox{for }i\in \mathcal{N}\left(y^t\right)\cap a_{10}\left(\bar x, \bar y\right),\\ 
0,&\mbox{else.}
\end{array}\right.\\
\]
Hence, $\left(x^t,y^t\right)$ fulfills (\ref{eq:kkt-1}). Also it is straightforward to check that (\ref{eq:kkt-2}) and (\ref{eq:kkt-3}) are fulfilled. 
Thus, $\left(x^t,y^t\right)$ is a Karush-Kuhn-Tucker point of $\mathcal{S}$. Moreover, ND1 is satisfied as well in view of Theorem \ref{thm:mpoc-licq}. Similar to (\ref{eq:kkt-2}) and (\ref{eq:kkt-3}), ND2 holds at $\left(x^t,y^t\right)$.
It remains to show ND3, i.e.
\[\xi_k^T D^2 L^\mathcal{S}\left( x^t, y^t\right) \xi_k\ne 0\quad \mbox{for }k=1,\ldots,2n-\alpha_t^{\mathcal{S}},\]
where  $\xi_1,\ldots\xi_{2n-\alpha_t^{\mathcal{S}}}$ form a basis of
$\mathcal{T}^{\mathcal{S}}_{\left(x^t,y^t\right)}$, cf. Step 1a from the proof of Theorem \ref{thm:regul}. Note, that by construction $\alpha_t^{\mathcal{S}}$ is constant for $t$ sufficiently small. Thus, we refer to it as $\alpha^{\mathcal{S}}$.
Next, we construct for $t>0$ sufficiently small such a basis as follows. 
First, we choose eigenvectors $\bar \xi_1,\ldots,\bar \xi_{2n-\alpha^{\mathcal{R}}}$ of $D^2 L^\mathcal{R}\left(\bar x, \bar y\right)$ forming a basis of $\mathcal{T}^{\mathcal{R}}_{\left(\bar x,\bar y\right)}$, cf. Step 1b from the proof of Theorem \ref{thm:regul}.
With similar arguments as in Step 2b of the proof of Theorem \ref{thm:regul} and by using the implicit function theorem, we find $\xi_j^t \in \mathcal{T}^{\mathcal{S}}_{\left(x^t,y^t\right)}$,$j=1,\ldots,2n-\alpha^{\mathcal{R}}$, still linearly independent. The remaining $2n-\alpha^{\mathcal{S}}-\left(2n-\alpha^{\mathcal{R}}\right)=\left\vert a_{00}\left(\bar x,\bar y\right)\right\vert$ vectors are chosen as follows. Namely, for $\bar i\in a_{00}(\bar x,\bar y)$ we consider 
\[
\mathcal{T}^{\bar i}_t=\left\{
\xi \in \R^{2n}\,\left\vert\, \begin{array}{l} \begin{pmatrix}
Dh_p(x^t),0\end{pmatrix} \xi=0, p \in P, 
\begin{pmatrix}
Dg_q(x^t),0\end{pmatrix}\xi=0,q \in Q_0(\bar x),\\
\begin{pmatrix}
0,e
\end{pmatrix}\xi=0,
\begin{pmatrix}
0,e_i
\end{pmatrix}\xi=0, i\in \mathcal{E}(\bar y),\\
\begin{pmatrix}
e_i,0
\end{pmatrix}\xi=0, i \in a_{00}(\bar x,\bar y)\backslash\{\bar i\},
\begin{pmatrix}
0,e_i
\end{pmatrix}\xi=0, i \in a_{00}(\bar x,\bar y)\backslash\{\bar i\},\\
\begin{pmatrix}
e_{\bar i},\frac{\varrho^t_{2,\bar i}}{\varrho^t_{1,\bar i}}{e_{\bar i}}
\end{pmatrix}\xi=0,
\begin{pmatrix}
y^t_ie_i,x^t_ie_i
\end{pmatrix}\xi=0, i \in a_{01}(\bar x,\bar y)\cup a_{10}(\bar x,\bar y)
\end{array}
\right.\right\}.
\]
As in Step 4 of the proof of Theorem \ref{thm:regul}, we can find $\bar \xi_{\bar i} \in \mathcal{T}^{\bar i}_0\backslash \mathcal{T}^{\mathcal{R}}_{\left(\bar x,\bar y\right)}$. Especially, $\left(e_{\bar i}, 0\right) \bar \xi_{\bar i} \ne 0$.
We note that $\lim\limits_{t\to 0} \frac{x^t_{\bar i}}{y^t_{\bar i}}=\frac{\bar \varrho_{2,\bar i}}{\bar \varrho_{1, \bar i}}$. Using this and the implicit function theorem, we find for $t$ sufficiently small a vector $\xi_{\bar i}^t \in \mathcal{T}^{\mathcal{S}}_{\left(x^t,y^t\right)}$.
It is then straightforward to check, that $\xi_j^t,j \in \left\{1,\ldots,2n-\alpha^{\mathcal{R}}\right\}$ and $\xi_{\bar i}^t$, $\bar i \in a_{00}(\bar x, \bar y)$, indeed form a basis of $\mathcal{T}^{\mathcal{S}}_{\left(x^t,y^t\right)}$.

We continue by considering the following limits with respect to the subsets of  
$\mathcal{H}\left(x^t,y^t\right)$ for any sequence of vectors $\xi^t \in \mathcal{T}^{\mathcal{S}}_{\left(x^t,y^t\right)}$ from the constructed base, cf. the definition of $\eta$-multipliers:
\[
 \begin{array}{lclcl}
 \displaystyle \lim\limits_{t \to 0} \sum\limits_{i \in a^<_{01}\left(\bar x, \bar y\right)} \eta^{\le,t}_i\xi^t_i\xi^t_{n+i}
 &=& \displaystyle \lim\limits_{t \to 0} \sum\limits_{i \in a^<_{01}\left(\bar x, \bar y\right)} \sigma_{1,i}^t\frac{1}{\bar y_i}\frac{x^t_i}{y^t_i}\left(\xi^t_{n+i}\right)^2
&=&0,
\\ \displaystyle
 \lim\limits_{t \to 0} \sum\limits_{i \in a^>_{01}\left(\bar x, \bar y\right)} \eta^{\ge,t}_i\xi_i^t\xi^t_{n+i}
 &=& \displaystyle \lim\limits_{t \to 0} \sum\limits_{i \in a^>_{01}\left(\bar x, \bar y\right)} -\sigma_{1,i}^t\frac{1}{\bar y_i}\frac{x^t_i}{y^t_i}\left(\xi^t_{n+i}\right)^2
 &=&0,
\\ \displaystyle
 \lim\limits_{t \to 0} \sum\limits_{i \in a^<_{10}\left(\bar x, \bar y\right), x^t_{i}>0} \eta^{\le,t}_i\xi^t_i\xi^t_{n+i}
 &=& \displaystyle \lim\limits_{t \to 0} \sum\limits_{i \in a^<_{10}\left(\bar x, \bar y\right), x^t_{i}>0} \sigma_{2,i}^t\frac{1}{\bar x_i}\frac{y^t_i}{x^t_i}\left(\xi^t_{i}\right)^2 
 &=&0,
\\ \displaystyle
 \lim\limits_{t \to 0} \sum\limits_{i \in a^<_{10}\left(\bar x, \bar y\right), x^t_{i}<0} \eta^{\ge,t}_i\xi^t_i\xi^t_{n+i}
 &=& \displaystyle \lim\limits_{t \to 0} \sum\limits_{i \in a^<_{10}\left(\bar x, \bar y\right), x^t_{i}<0} -\sigma_{2,i}^t\frac{1}{\bar x_i}\frac{y^t_i}{x^t_i}\left(\xi^t_{i}\right)^2 
 &=&0.
 \end{array}
\]
If for some $i \in a^<_{00}(\bar x, \bar y)$ or $i \in a^>_{00}(\bar x, \bar y)$ it holds $\lim\limits_{t\to 0}\left(e_{i}, 0\right) \xi^t\ne 0$, we observe:
\[
 \begin{array}{lclcl}
 \displaystyle \lim\limits_{t \to 0} \sum\limits_{i \in a^<_{00}\left(\bar x, \bar y\right)} \eta^{\le,t}_i\xi^t_i\xi^t_{n+i}
 &=& \displaystyle\lim\limits_{t \to 0} \sum\limits_{i \in a^<_{00}\left(\bar x, \bar y\right)} -\frac{\sqrt{\varrho_{1,i}^t\varrho_{2,i}^t}}{\sqrt{t}}\frac{\varrho_{1,i}^t}{\varrho_{2,i}^t}\left(\xi^t_{i}\right)^2
 &=&-\infty,
\\ \displaystyle \lim\limits_{t \to 0} \sum\limits_{i \in a^>_{00}\left(\bar x, \bar y\right)} \eta^{\ge,t}_i\xi^t_i\xi^t_{n+i} &=&\displaystyle \lim\limits_{t \to 0} \sum\limits_{i \in a^>_{00}\left(\bar x, \bar y\right)} -\frac{\sqrt{-\varrho_{1,i}^t\varrho_{2,i}^t}}{\sqrt{t}}\frac{\varrho_{1,i}^t}{\varrho_{2,i}^t}\left(\xi^t_{i}\right)^2
 &=&\infty.
 \end{array}
\]
Finally, we calculate as in (\ref{eq:s-r-h1}):
\[
\label{eq:s-r-h2}
\begin{array}{rcl}
\left(\xi^t\right)^T D^2 L^\mathcal{S}\left( x^t, y^t\right) \xi^t&=& \displaystyle
 \left(\xi^t\right)^T D^2 L^\mathcal{R}\left(x^t, y^t\right) \xi^t - \sum\limits_{i\in \mathcal{H}^\ge\left(x^t,y^t\right)} 2\eta_i^{\ge,t} \xi^t_i  \xi^t_{n+i} \\ && \displaystyle +
\sum\limits_{i\in \mathcal{H}^\le\left(x^t,y^t\right)} 2\eta_i^{\le,t}  \xi^t_i \xi^t_{n+i}.
\end{array}
\]
Altogether, we obtain for any basis vector $\xi_j^t,j \in \left\{1,\ldots, 2n-\alpha^{\mathcal{R}}\right\}$ of $\mathcal{T}^{\mathcal{S}}_{\left(x^t,y^t\right)}$:
\begin{equation}
    \label{eq:eig-j}
\lim\limits_{t \to 0}\left(\xi_j^t\right)^T D^2 L^\mathcal{S}\left( x^t, y^t\right) \xi_j^t
=\bar a_j\left\lVert \bar \xi_j \right\rVert^2,\end{equation}
where $\bar a_j$ is a nonzero eigenvalue of $D^2 L^\mathcal{R}\left(\bar x, \bar y\right)$ due to the choice of $\bar \xi_j$ and NDT4. 
Let us now focus on the basis vectors $\xi^t_{\bar i}$, $\bar i \in a_{00}(\bar x, \bar y)$.
We then have: 
\begin{equation}
    \label{eq:eig-i00}
\lim\limits_{t \to 0}\left(\xi_{\bar i}^t\right)^T D^2 L^\mathcal{S}\left( x^t, y^t\right) \xi_{\bar i}^t
=-\infty.
\end{equation}
This is due to the following reasoning. First,
$\left(\xi_{\bar i}^t\right)^T D^2 L^\mathcal{R}\left( x^t, y^t\right) \xi_{\bar i}^t$ is bounded due to the construction of $\xi_{\bar i}^t$, and, moreover, $\lim\limits_{t\to 0}\left(e_{\bar i}, 0\right) \xi^t_{{\bar i}}\ne 0$.
We conclude that ND3 is fulfilled. Additionally, the T-index of $(\bar x, \bar y)$ is equal to the sum of its quadratic index and its biactive index, i.e. $m = \overline{QI}^{\mathcal{R}}+\left\vert a_{00}(\bar x,\bar y)\right\vert $. In view of (\ref{eq:eig-j}) and (\ref{eq:eig-i00}), the quadratic index ${QI}^\mathcal{S}_t$ of $(x^t,y^t)$ is then exactly $m$ for $t$ sufficiently small. 
\end{proof}

\section*{Conclusions}

In \cite{laemmel:reform}, the number of saddle points for the regularized continuous reformulation of CCOP has been estmated. Namely, each saddle point of CCOP generates exponentially many saddle points of $\mathcal{R}$, all of them having the same index. It has been concluded there that the introduction of auxiliary $y$-variables shifts the
complexity of dealing with the cardinality constraint in CCOP into the appearance of multiple saddle points for its continuous reformulation. From our extended convergence analysis of the Scholtes-type regularization it follows that the number of its saddle points also grows exponentially as compared to that of CCOP. We emphasize that this issue is at the core of numerical difficulties if solving CCOP up to global optimality by means of the Scholtes-type regularization method. 
To the best of our knowledge this is the first paper studying convergence properties of the Scholtes-type regularization method in the vicinity of saddle points, rather than of minimizers. The ideas from our analysis can be potentially applied not only for other classes of nonsmooth optimization problems, such as MPCC, MPVC, MPSC, and MPOC, but also for other regularization schemes known from the literature. 

\bibliography{sn-bibliography.bib}

\end{document}